\documentclass[final,3p,times]{elsarticle}
\usepackage{amssymb,amsmath,amsthm,bm}
\usepackage{epstopdf}   % trying to auto generate PDF from EPS

\newtheorem{defn}{Definition}
\newtheorem{thm}{Theorem}
\newtheorem{lem}{Lemma}
\newtheorem{prop}{Proposition}
\newtheorem{rmk}{Remark}

\newcommand{\Sing}{\mathcal{S}}
\newcommand{\Doub}{\mathcal{D}}

\newcommand{\R}{\mathbb{R}}

\newcommand{\bi}{\begin{itemize}}
\newcommand{\ei}{\end{itemize}}
\newcommand{\ben}{\begin{enumerate}}
\newcommand{\een}{\end{enumerate}}
\newcommand{\be}{\begin{equation}}
\newcommand{\ee}{\end{equation}}
\newcommand{\bea}{\begin{eqnarray}} 
\newcommand{\eea}{\end{eqnarray}}
\newcommand{\ba}{\begin{align}} 
\newcommand{\ea}{\end{align}}
\newcommand{\bse}{\begin{subequations}} 
\newcommand{\ese}{\end{subequations}}
\newcommand{\bc}{\begin{center}}
\newcommand{\ec}{\end{center}}
\newcommand{\bfi}{\begin{figure}}
\newcommand{\efi}{\end{figure}}
\newcommand{\ca}[2]{\caption{#1 \label{#2}}}

\newcommand{\ig}[2]{\includegraphics[#1]{#2}}
\newcommand{\tbox}[1]{{\mbox{\tiny #1}}}
\newcommand{\mbf}[1]{{\mathbf #1}}
\newcommand{\pO}{{\partial\Omega}}

\newcommand{\eO}{{\R^2\backslash\overline{\Omega}}}   % open exterior domain
\newcommand{\half}{\mbox{\small $\frac{1}{2}$}}
\newcommand{\sfrac}[2]{\mbox{\small $\frac{#1}{#2}$}}
\DeclareMathOperator{\re}{Re}
\DeclareMathOperator{\im}{Im}
\DeclareMathOperator{\Ai}{Ai}
\DeclareMathOperator{\Bi}{Bi}
\newcommand{\ui}{u^\tbox{inc}}
\newcommand{\xx}{\mbf{x}}
\newcommand{\yy}{\mbf{y}}
\newcommand{\zz}{\mbf{z}}
\newcommand{\nn}{\mbf{n}}
\newcommand{\kk}{\mbf{k}}
\newcommand{\pp}{\mbf{p}}

\newcommand{\xp}{\mbf{x}'}
\newcommand{\yp}{\mbf{y}'}
\newcommand{\kp}{\mbf{k}'}

\newcommand{\infint}{\int_{-\infty}^{\infty}}      % infinite integral
\newcommand{\fref}[1]{Fig.~\ref{#1}}          % figure ref, like \eqref
\newcommand{\sref}[1]{Sec.~\ref{#1}}          % Section ref, like \eqref
\newcommand{\tref}[1]{Table~\ref{#1}}
\newcommand{\dref}[1]{Definition~\ref{#1}}
\newcommand{\rref}[1]{Remark~\ref{#1}}

\newcommand{\al}{\alpha}               % contour parameter
\newcommand{\eps}{\varepsilon}

\journal{Journal of Computational Physics}

\begin{document}

\begin{frontmatter}



\title{High-order boundary integral equation solution of high frequency wave
scattering from obstacles in an unbounded linearly stratified medium}
% stratified?
% high-order accurate? Not novel


\author[d]{Alex H. Barnett}
\author[s]{Bradley J. Nelson}
\author[v]{J. Matthew Mahoney}
\address[d]{Department of Mathematics, Dartmouth College, Hanover, NH, 03755}
\address[s]{Institute for Computational and Mathematical Engineering, Stanford University, Stanford, CA, 94305}
\address[v]{Department of Neurological Sciences, University of Vermont, Burlington, VT, 05405}

\begin{abstract}
We apply boundary integral equations for the first time to the
two-dimensional scattering of time-harmonic waves from a smooth obstacle
embedded in a continuously-graded unbounded medium.
In the case we solve the square of the wavenumber
(refractive index) varies linearly
in one coordinate, i.e. $(\Delta + E + x_2)u(x_1,x_2) = 0$
where $E$ is a constant;
this models quantum particles of fixed energy
in a uniform gravitational field,
and has broader applications to stratified media in
acoustics, optics and seismology.
We evaluate the fundamental solution efficiently with exponential accuracy via
numerical saddle-point integration, %applied to a contour integral,
using the truncated trapezoid rule with typically $10^2$ nodes,
with an effort that is {\em independent} of the frequency parameter $E$.
By combining with
high-order Nystr\"om quadrature, we are able to solve the scattering
from obstacles 50 wavelengths across to 11 digits of accuracy in under a minute
on a desktop or laptop.
\end{abstract}


\end{frontmatter}


%\section{}
%\label{}

\section{Introduction}
Problems involving
time-harmonic waves in media whose wave speed or refractive
index varies continuously in a layered fashion are common in both the natural
and engineered worlds.
In acoustics, underwater sound propagation
\cite{kellerunderwater,etter},
and environmental noise modeling in the presence
of a thermal gradient \cite{meteobem}
both involve continuously stratified wave speeds.
In electromagnetics,
continuously stratified media occur in
ionospheric propagation \cite{hartree}
and nano-scale optical devices (see \cite{nanofilms} and references within).
%with a discontinuous layered variation in refractive index
%are common \cite{tamir}, but graded index optics is common 
In elastodynamics,
similar models play important roles in seismology since
wave speed grows in a piecewise continuous fashion with
with depth into the earth \cite[Sec.~2.5.3]{chapman},
and in designing functionally graded materials \cite{suresh}.
In quantum physics the same equations as in acoustics arise when
gravitational or electric fields
influence the motion of fixed energy particles \cite{bracher}.
In each case, when the varying medium is acoustically large (many wavelengths 
across), or unbounded,
accurate numerical solution of wave propagation and scattering
remains challenging.

We will solve the following scalar-wave exterior
boundary value problem (BVP),
where $\Omega\subset\R^2$ is a given
bounded obstacle with smooth boundary $\pO$,
and $f$ is smooth Dirichlet data on $\pO$,
\bea
\bigl(\Delta + k(x_2)^2\bigr) u(x_1,x_2) & = & 0
\qquad \xx:=(x_1,x_2) \in \eO~,
\label{lhelm}
\\
u &=& f \qquad \mbox{ on } \pO~,
\label{bc}
\eea
where $\Delta:=\partial^2/\partial x_1^2 +\partial^2/\partial x_2^2$
is the Laplace operator, with the specific vertical wavenumber variation
$k(x_2)$ given by
\be
k(x_2)^2 = E + x_2
~,
\label{ourk}
\ee
and outgoing radiation conditions for $u$.
The latter, given in Definition~\ref{d:rbcs},
are required for uniqueness of the solution.
In applications the potential $u$ represents
pressure, wavefunction, or a component of electric or magnetic field.

The general relationship $k = \omega/c$, where $\omega$ is frequency
and $c$ wave speed, means that in the frequency-domain (fixed $\omega$)
case, $k$ is proportional to the {\em refractive index}
and inversely proportional to the wave speed.
In \eqref{ourk} the inverse square of
wave speed (sometimes called {\em sloth}) is linear in the
vertical ($x_2$) coordinate, a model found in seismology
\cite[Sec.~2.5.2.2]{chapman}; in the electromagnetic case
\eqref{ourk} corresponds to linear variation in permittivity
\cite[Sec.~2.5.1]{chewbook}.
% rephrase?
Recalling that the Helmholtz equation $(\Delta + E)u = 0$
models free-space quantum particles at energy $E$,
%(or acoustics at frequency $E$), - nope
we call \eqref{lhelm} with \eqref{ourk} the ``gravity Helmholtz equation''
because it is a non-dimensionalized%
\footnote{We chose a unity constant in front of $x_2$ without loss of generality
since adjusting this constant is equivalent to rescaling the domain $\Omega$.}
model for quantum particles at energy $E$
in a uniform gravitational \cite{bracher} or electric \cite{gottlieb} field,
i.e.\ a linear potential.
Its one-dimensional (1D) solution is the Airy function, and its
application goes back at least to Hartree's 1931 work on the
ionosphere \cite[Sec.~6]{hartree}.
The constant $E$ sets the square of the wavenumber at the height
$x_2=0$; the waves have evanescent (modified Helmholtz) character for $x_2<-E$,
changing to
oscillatory (Helmholtz) character for $x_2>-E$.
The asymptotic behavior of solutions to \eqref{lhelm} is radically different
in the horizontal and vertical directions, with waves eventually
``dragged'' into a narrow upwards-propagating beam; see \fref{f:fs}(b).
In the optical and acoustic setting, the imaginary refractive index
for $x_2<-E$ could be relevant for graded metamaterials,
although a more common application of the effficient PDE solver
we present might be to acoustic or electromagnetic propagation in
subregions of the plane (a half-space, etc).
%We believe that the PDE might also be a useful approximation
%in other applications such as acoustic propagation in a thermal gradient.

In the usual setting of scattering theory (see \fref{f:geom})
 an incident wave $\ui$
satisfying \eqref{lhelm} in the entire plane impinges on the obstacle;
the scattered wave is then $u$, the 
solution to the above exterior BVP with boundary data $f = -\ui$ on $\pO$.
The physical potential is then $\ui + u$.
The Dirichlet case we study corresponds to sound-soft acoustics,
or $z$-invariant Maxwell's equations with a perfect electric conductor
in transverse-magnetic polarization.
The Neumann (sound-hard) case can be solved with similar tools \cite{kress95}.
We will also solve the interior Dirichlet BVP, with applications to
graded-index optics,
%and the ``equivalent profile'' method for solving
and to transverse acoustic or optical
modes in a bending waveguide in 3D
approximated by an ``equivalent profile''
in which the square of refractive index varies linearly
\cite{marcuse}.

We propose boundary integral equations (BIE) as an efficient and
accurate numerical method to solve \eqref{lhelm}--\eqref{bc}.
This demands being able to compute values and first derivatives 
of $\Phi(\xx,\yy)$, the {\em fundamental solution}
to \eqref{lhelm}, where $\xx,\yy \in\R^2$ are target and source points
respectively. Recall the definition that, for
a source point $\yy\in\R^2$, $\Phi(\cdot,\yy)$ is the radiative solution to the PDE
\be
-(\Delta_\xx +  k(x_2)^2) \Phi(\xx,\yy) = \delta(\xx-\yy)~,
\label{Phidef}
\ee
where $\delta$ is the Dirac delta distribution in $\R^2$.
In contrast to the common situation, $\Phi$ is
no longer an elementary or special function of distance $|\xx-\yy|$;
this is clear in \fref{f:fs}.
A large part of our contribution is an efficient numerical method for evaluation
of $\Phi$, by applying quadrature to %an integral formula arising from
the Fourier transform of an analytical
solution to the time-dependent
Schr\"odinger equation in a linear potential \cite{hellerhouches}.
% the physics community.
Unfortunately the integral is highly oscillatory, especially as $E$ grows,
thus we use deformation of the contour into the complex plane,
passing through the saddle (stationary phase) points
%\cite[Sec.~5.5]{temme},
%i.e.\ numerical steepest descent \cite{huy},
and using the trapezoid rule \cite{PTRtref} to
achieve exponential accuracy with effort independent of $E$.
The saddle points will have an elegant interpretion as the
classical ray travel times.
The cost of each evaluation of $\Phi$ is only a few hundred complex
exponential evaluations,
hence we achieve typically $10^5$ evaluations per second.

\bfi
\hspace{1in}\ig{width=4.5in}{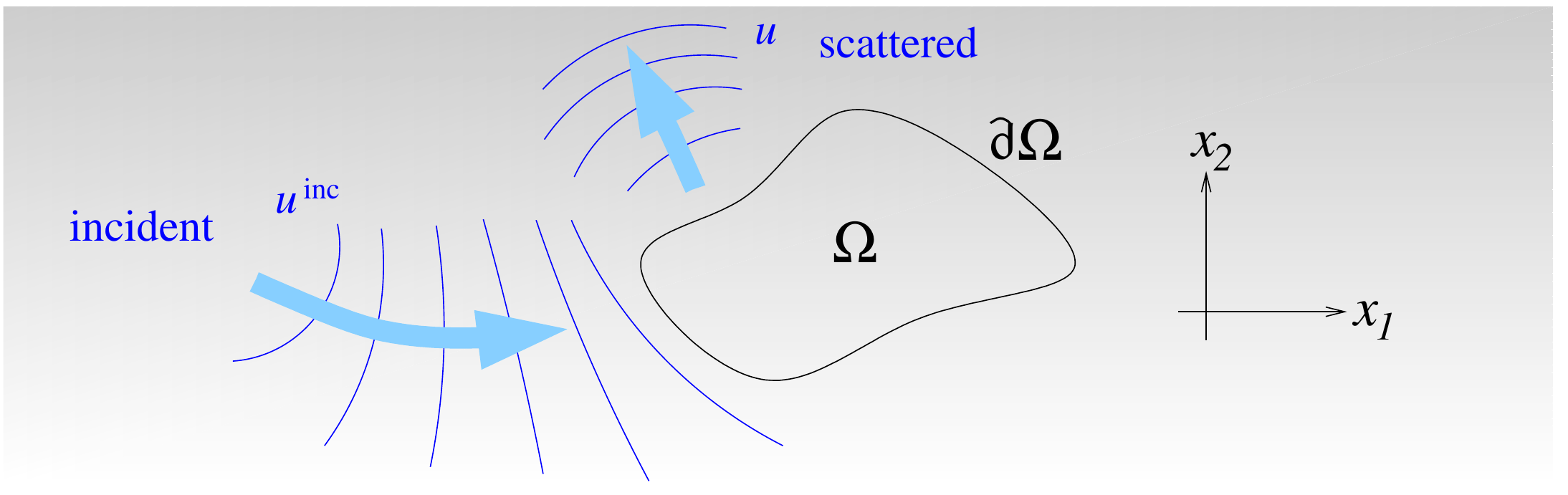}
\ca{Geometry for the scattering problem embedded in a stratified medium.
Wave speed decreases (refractive index increases) in the vertical $x_2$ direction.}{f:geom}
\efi

\bfi % fffffffffffffffffffffffffffffffffffffffffffffffffffffffffffff
\ig{width=2.2in}{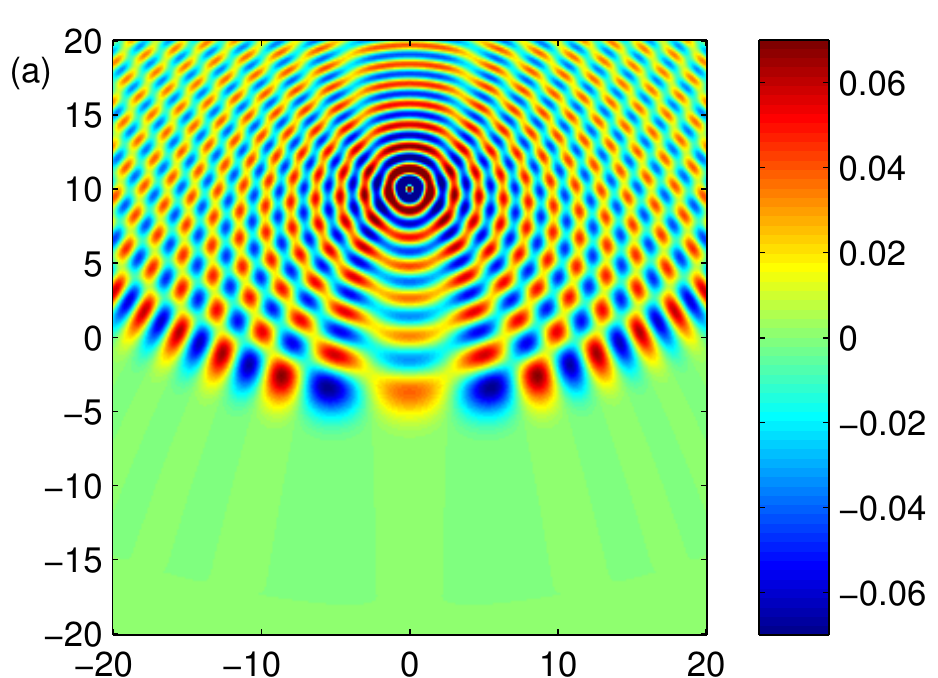}
\ig{width=2.2in}{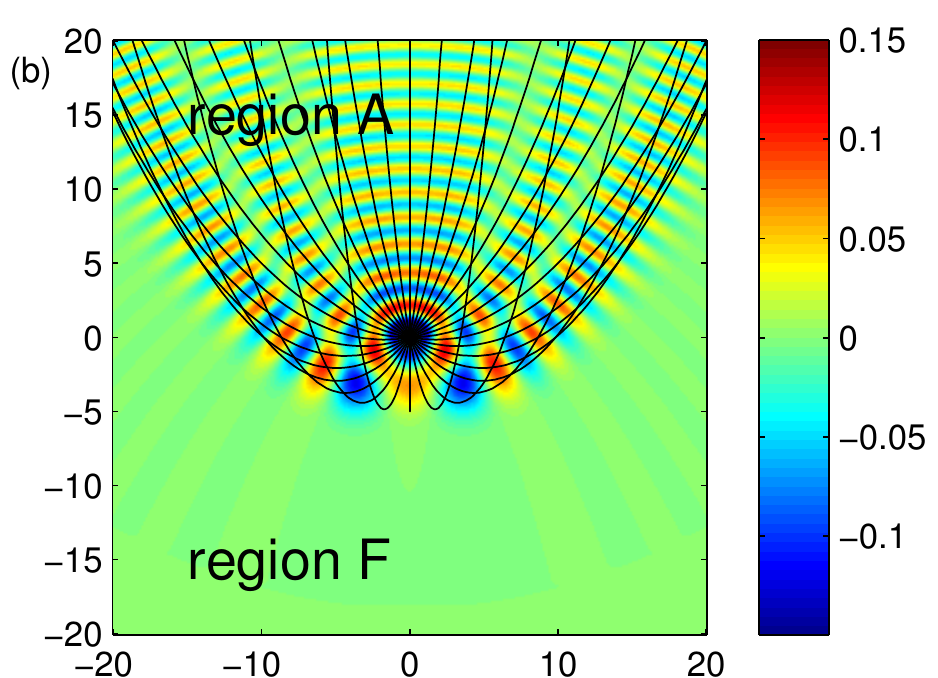}
\ig{width=2.2in}{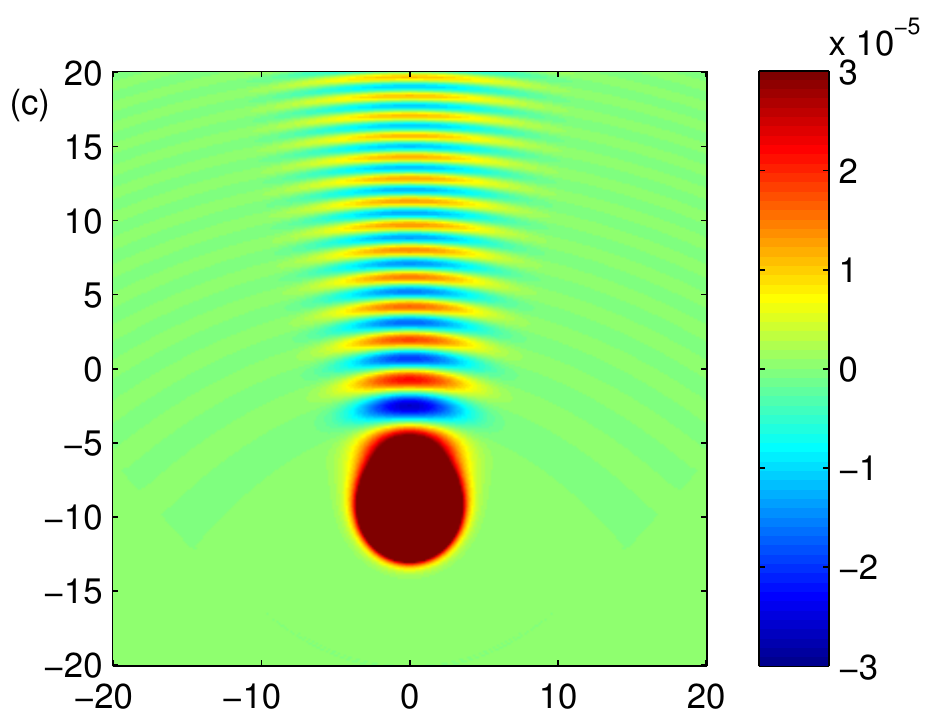}
\caption{Real part of fundamental solution $\Phi(\cdot,\yy)$ plotted in
$\R^2$ for the case $E=5$ and three choices of source location $\yy$:
(a) $\yy = (0,10)$; (b) $\yy = (0,0)$; (c) $\yy = (0,-10)$.
In (b) we also show parabolic
classical ray trajectories emanating from the source $\yy$,
which themselves all lie within region A
(a parabola with focus $\yy$), discussed in Sec.~\ref{s:ray}.
Notice the color scale in (c) indicating the very small amplitude of the
propagating beam.
\label{f:fs}
}
\efi % fffffffffffffffffffffffffffffffffffffffffffffffffffffff

\subsection{Relation to previous work on frequency domain wave propagation in
layered media}

%We now overview previous work on such layered media frequency-domain
%wave problems.
Accurate numerical propagation of high frequency waves in a
variable medium is numerically challenging:
conventional ``volume''
discretization methods such as finite differencing (FD)
\cite{leveque}
and finite elements (FEM)
%are low order and
require several degrees of freedom per wavelength
to achieve reasonable accuracy;
moreover, in order to avoid ``pollution errors''
the degrees of freedom per wavelength must grow with frequency
\cite{pollution}.
The resulting linear systems are so huge that iterative solvers are almost always used,
and yet preconditioning has mostly been unsuccessful for the
high-frequency Helmholtz equation,
especially for high-order discretizations,
and is a topic of current
research \cite{2011_engquist_ying_PML}.
The radiation condition must still be approximated via
artificial absorbing boundary conditions
(e.g.\ perfectly matched layers) \cite{ABC_1977} \cite[Sec.~4.7]{chewbook}.
% mention FDS \cite{iti} here?

At high frequencies, ray approximation is useful \cite{chapman}
and geometric diffraction theory can approximate the
interaction with simple obstacles.
However, such approximations break down at turning points
(such as at $x_2=-E$) and for geometric details
on the wavelength scale.
Parabolic approximation (i.e.\ one-way wave equation) methods \cite{oneway}
handle only a limited range of propagation directions,
and cannot account for back reflections.
% or turning rays.
Several of these methods are reviewed in the underwater acoustic
and elastic contexts in \cite{waveprop}.

When the medium (PDE coefficient) is piecewise constant,
reformulation as a boundary integral equation (BIE)
\cite[Ch.~3]{coltonkress} \cite{atkinson} \cite[Ch.~8]{chewbook}
is popular due to several advantages:
\bi
\item
the unknowns live on the boundary (or material interfaces)
rather than the volume; this reduction in dimension by one
greatly reduces the number of unknowns $N$,
especially at high frequencies,
and simplifies the geometric issues (meshing, etc);
\item
when a second-kind formulation is used, it remains {\em well-conditioned}
(and hence iterative methods rapidly convergent)
independent of the number of discretization nodes used;
\item
radiation conditions are already built into the representation and
need not be enforced, unlike in FD or FEM;
\item
fast algorithms such as the fast multipole method \cite{fmm2}
or fast direct solvers \cite{hackbusch,m2011_1D_survey} can reduce
the solution time to $O(N)$ for low frequencies;
\item
in the two-dimensional case, high-order quadratures on boundary curves
are easy to implement \cite{kress91,hao}.
\ei
This has enabled the scattering from objects %(embedded
(in a uniform medium) thousands of wavelengths across to be solved efficiently
to many digits of accuracy (e.g.\ see \cite{Mart07}).

In contrast, we care about scattering in a continuously-varying %unbounded
medium.
If this medium were constant outside a bounded region,
a Lippmann--Schwinger (volume integral) equation
\cite[Ch.~8]{coltonkress} could be used,
or coupling of direct discretization methods to BIE \cite{kirsch,iti}.
%
%recently-developed methods coupling fast direct solvers which couple 
%\cite{iti}.
%however, the number of unknowns is large, comparable to that needed
%for FD or FEM.
%
Tools also exist for BIE within media with a
finite number of constant layers \cite{chocai12}.
The method of the present paper extends
%the efficiency and
the above advantages of BIEs
to a particular problem where the stratified medium variation---and the
resulting wave propagation---is
smooth and {\em unbounded} in all directions.
We are not aware of previous applications of BIE to such a case.
%For our problem the Green's function may be computed
%to spectral accuracy reasonably fast.
%
The only similar work we know of is that of
Premat--Gabillet in their environmental acoustics
code Meteo-BEM \cite{meteobem},
who use BIEs with the Green's function for a linear wave speed profile.
However, they approximate the Green's function using a discrete sum over
1D eigenfunctions, an approach that works only when waves
are trapped by a ground plane; this would fail in the case of unbounded
propagation.
Also, since their BIE is of Fredholm first kind, the convergence rate
of an iterative solver would be poor.
%well-conditioned
%and would be unable to solve problems iteratively
%(eg via GMRES) due to the poor eigenvalue distribution of the matrix.
%
\begin{rmk}  % rrrrrrrrrrrrrrrrrrr
Our approach to evaluate the Green's function is reminiscent
of the Sommerfeld integral
(spectral representation) commonly
used for layered media \cite[Ch.~2]{chewbook}, \cite{chocai12}.
Yet, although both methods exploit numerical quadrature of a contour integral,
%but in different variables.
they are distinct, with crucial differences.
In the Sommerfeld approach the integration variable is a {\em transverse
wavenumber}, and a vertical ODE has to be solved for each
contour quadrature node; for the profile \eqref{ourk} this
would demand Airy functions.
The number of quadrature nodes needed grows linearly with wavenumber,
for fixed source-target separation.
In addition, the decay of the Sommerfeld integrand is known to be very slow when
the vertical %source-target
separation is small,
demanding various windowing approximations \cite{chocai12}.
In contrast, in
our proposed scheme the integration variable represents {\em time},
the integrand involves only exponentials, and
by choosing appropriate complex contours the number of nodes
is independent of wavenumber.
Of course, the Sommerfeld approach has the advantage over our scheme that,
assuming the ODEs can be solved fast enough,
arbitrary profiles $k(x_2)$ could be handled.
\end{rmk}  % rrrrrrrrrrrrrrrrrrrrrr

%\subsection{Radiation conditions for the BVP}

%Although much work has been done 

%Generalizations:
%Incident ``plane wave'' (Airy) is also possible.
%3D.
%eigenmodes.
%interior graded regions.

\subsection{Outline of the paper}

We use the remainder of the introduction to state a radiation
condition that allows a unique solution to our BVP
(this is proved in Appendix~A).
In \sref{s:fs} we present an integral formula for the
fundamental solution \eqref{Phidef} for the PDE \eqref{lhelm};
here the radiation condition derives from causality in the time domain.
We then use potential theory to reformulate the BVP as
an integral equation on $\pO$ in \sref{s:bie},
and present its high-order numerical solution, which demands many
evaluations of the fundamental solution.
\sref{s:eval} is the key part of the paper in which
we present efficient new contour quadrature algorithms for this task.
In \sref{s:num} we present numerical tests of convergence and speed
for both the interior and exterior BVPs.
We draw some conclusions and discuss future work in \sref{s:conc}.

% RRRRRRRRRRRRRRRRRRRRRRRRRRRRRRRRRRRRRRRRRRRRRRRRRRRRRRRRRRRRRRRRRRRRRRRRRRRR
\subsection{The radiation condition for the BVP}
\label{s:rc}

Recall that for the constant-$k$ Helmholtz equation $(\Delta + k^2)u = 0$
in $\R^2$, the Sommerfeld radiation condition \cite[(3.62)]{coltonkress}
is $\partial u/\partial r - iku = o(r^{-1/2})$,
holding uniformly in angle, where $r := |\xx|$.
This corresponds to outgoing %(rather than incoming)
waves at infinity.
It guarantees a unique solution
to exterior BVPs \cite[Sec.~3.2]{coltonkress},
for instance the case of Dirichlet data \eqref{bc}.
Radiation conditions are also known for stratified media that are eventually
constant or tend to a constant in
upper and lower half-planes \cite{wilcoxstrat,jeresRC},
for variable media that tend towards a constant at large
distances \cite{miranker}, and for scattering from
unbounded rough surfaces in a uniform medium \cite{UPRC}.
For our exterior gravity Helmholtz equation there are no trapped waveguide
modes because the refractive index is monotonic in $x_2$,
simplifying the situation from that of \cite{wilcoxstrat,jeresRC}.
And yet, we have not been able to find a radiation condition
in the literature that applies in our case where
the wavenumber is unbounded in one direction.

Hence we propose the following new radiation condition,
recalling the notation $\xx=(x_1,x_2)$.
\begin{defn}[Radiation condition]  % ddddddddddddddddddddddddddddd
A solution $u$ to \eqref{lhelm} in the exterior of a bounded domain
$\Omega \subset \R^2$, with medium defined by \eqref{ourk},
is called {\em radiative} if
\bea
\lim_{x_2 \to +\infty} \frac{1}{k(x_2)} \infint
\left|\frac{\partial u}{\partial x_2} - i k(x_2) u \right|^2 dx_1 &=& 0
\label{UPRC}
\\
\lim_{x_2 \to -\infty} \infint |u|^2 + \left|\frac{\partial u}{\partial x_2}\right|^2 dx_1 &=& 0
\label{DRC}
\\
\lim_{L\to\infty} \lim_{x_1 \to \pm\infty} \int_{-L}^L |u|^2 + \left|\frac{\partial u}{\partial x_1}\right|^2 dx_2 &=& 0
\label{LRRC}
\eea
\label{d:rbcs}
\end{defn}           % dddddddddddddddddddddddddddddddddddddddd
The first condition states that the flux is eventually
upwards-going on positive horizontal slices;
the other two guarantee enough decay that the flux tends to zero
on the sides and bottom of a large rectangular box.
Note that these conditions could most likely be tightened;
however, they are adequate for our purpose, namely
%In practical settings the unbounded growth of the wavenumber
%is unlikely.
to prove in Appendix~\ref{a:rc} the following uniqueness result,
analogous to \cite[Thm.~3.7]{coltonkress} for the Helmholtz equation.
This places our BVP on a more rigorous footing.
\begin{thm}
There is either zero or one
radiative exterior solution to \eqref{lhelm}--\eqref{ourk}.
\label{t:rc}
\end{thm}

% FFFFFFFFFFFFFFFFFFFFFFFFFFFFFFFFFFFFFFFFFFFFFFFFFFFFFFFFFFFFFFFFFFFFFFFFFFFF
\section{The fundamental solution and its ray interpretation}
\label{s:fs}

In this section we derive an integral formula for the fundamental
solution for our PDE \eqref{lhelm} in $\R^2$, and give some of its
properties.
In fact, since it requires no extra effort, we work in $\R^n$
and then specialize to $n=2$.
Let $\xx = (\xp,x_n)$ where $\xp=(x_1,\ldots,x_{n-1})$ is the
transverse coordinate and $x_n$ is the vertical one.
The gravity Helmholtz equation in $\R^n$ is
$(\Delta + E + x_n)u(\xx) = 0$.
Recall that the fundamental solution is defined by \eqref{Phidef}.
We will exploit causality in the time domain to
obtain a solution with physically correct radiation conditions,
so call this the ``causal'' fundamental solution
(although see \rref{r:exist}).

\begin{lem}[Bracher et al.~{\cite{bracher}}]
The %radiative
causal fundamental solution to the gravity Helmholtz equation
$(\Delta + E + x_n)u(\xx) = 0$ 
in $\R^n$ with source point $\yy\in\R^n$ is given by
\be
\Phi(\xx,\yy) = \frac{i}{(4\pi i)^{n/2}}
\int_0^\infty \frac{1}{t^{n/2}} \exp i\left[\frac{|\xx-\yy|^2}{4t} +
\left( \frac{x_n+y_n}{2} + E\right) t - \frac{1}{12} t^3
\right]
dt ~.
\label{Phiint}
\ee
\label{l:fs}
\end{lem}
Its proof exploits the fact that the time-dependent {\em Schr\"odinger
equation}
has an analytically known fundamental solution in a linear potential.
We will show that the integral in
\eqref{Phiint} is in fact the Fourier transform from time $t$ to energy $E$;
note that this is distinct from the more usual connection of frequency-domain
fundamental solutions to the {\em wave equation},
for instance in the Cagniard--de Hoop method \cite[Sec.~4.2]{chewbook}.
%The radiative nature will arise from the causality in time.

For convenience we
simplify and rephrase the derivation of Bracher et al. \cite{bracher} in a more
mathematical language, and in dimensionless units.
Our definitions of the Fourier transform from time to energy will be,
in terms of a general function $f$,
$$
\int_{-\infty}^\infty \tilde{f}(t) e^{iEt}dt = f(E)
~,
\hspace{1in}
\frac{1}{2\pi}\int_{-\infty}^\infty f(E) e^{-iEt}dE = \tilde{f}(t)
~.
$$
Similarly, our definition for spatial Fourier transforms is
$$
\int_{\R^n} \hat{f}(\kk) e^{i \kk \cdot\xx} d\kk = f(\xx)
~,
\hspace{1in}
\frac{1}{(2\pi)^n}\int_{\R^n} f(\xx) e^{-i \kk \cdot\xx} d\xx = \hat{f}(\kk)
~.
$$
We now prove the lemma.
\begin{proof}
We will isolate the last coordinate with the
notation $\xx=(\xp,x_n)$ and $\yy = (\yp,y_n)$.
Suppressing for now the $\yy$ dependence, but making the dependence on $E$
explicit, the fundamental solution obeys
$$
(\Delta + E + x_n)\Phi(\xp,x_n;E) = -\delta(\xx-\yy)
~.
$$
The Fourier transform from $E$ (energy) to $t$ (time) turns this into
\be
(\Delta + i\partial_t + x_n)\tilde{\Phi}(\xp,x_n;t) = -\delta(\xx-\yy)\delta(t)
\label{tdse}
\ee
which is the fundamental solution for the
time-dependent Schr\"odinger equation in a linear potential.
We may solve this exactly by performing a Fourier transform in space,
from coordinates $(\xp,x_n)$ to wavevector $(\kp,K)$,
$$
\left[-|\kp|^2 - K^2 + i(\partial_t + \partial_K) \right] \hat{\tilde{\Phi}}
(\kp,K;t) = -(2\pi)^{-n} e^{-i\kk\cdot\yy} \delta(t)
~.
$$
The only derivatives are an advection term causing constant unit
speed drift in wavevector in the positive $x_n$ direction,
so we shift to a frame moving in wavevector, substituting
$\kappa = K-t$
(in physics this is called a gauge change \cite[App.~A]{bracher}).
To change from coordinates $(\kp,K;t)$ to $(\kp,\kappa;t)$ we then need
$$ \left.\frac{\partial}{\partial t}\right|_\kappa = 
\left.\frac{\partial}{\partial t}\right|_K +
\frac{\partial}{\partial K}
~.
$$
This gives the simple first-order ODE in time at each wavevector
$(\kp,\kappa)\in\R^n$,
$$
\left[-|\kp|^2 - (\kappa+t)^2 + i\partial_t\right] \hat{\tilde{\Phi}}
(\kp,\kappa;t) = -(2\pi)^{-n} e^{-i(\kp\cdot\yp +\kappa y_n)}\delta(t)
~.
% note we dropped the exp -i y_n t  factor
$$
For each $(\kp,\kappa)\in\R^n$ we
seek a causal solution with $\hat{\tilde{\Phi}}(\kp,\kappa;t)=0$
for all $t<0$. The right-hand side is an impulsive excitation at $t=0$
which gives the ODE solution
$$
\hat{\tilde{\Phi}}(\kp,\kappa;t) = \frac{i e^{-i(\kp\cdot\yp +\kappa y_n)}}{(2\pi)^n}
\exp i\left[-|\kp|^2 t - \kappa^2 t - \kappa t^2 - \frac{1}{3}t^3  \right]
~,\qquad t>0 ~.
$$
Changing back to the original wavevector coordinates via $\kappa=K-t$ gives
$$
\hat{\tilde{\Phi}}(\kp,K;t) = \frac{ie^{i(y_nt- \frac{1}{3}t^3)}}{(2\pi)^n}
e^{-i \kk\cdot\yy}
\exp i\left[-|\kp|^2 t - K^2 t + K t^2 \right]
~,\qquad t>0 ~.
$$
The final exponential is an (imaginary) gaussian in Fourier space,
whose inverse spatial Fourier transform is known exactly.
The middle exponential term causes a real space translation by $\yy$.
% note it's subtle: the y_n t out front changes the sign of the y_n
% term in the spatial part!
This gives after simplification,
%$$
%\tilde{\Phi}(\xx;t) = \frac{-i}{(4\pi i t)^{n/2}}
%\exp i\left[\frac{|\xx|^2}{4t} + \frac{x_n}{2} t - \frac{1}{12}t^3
%\right]
%$$
%The case of general $\yy\in\R^n$ can be recovered by translational
%invariance, recognizing that an overall shift in the vertical
%$x_n$ direction corresponds to a shift in $E$.
%Thus we replace $\xx$ by $\xx-\yy$, giving
\be
\tilde{\Phi}(\xx,\yy;t) = \frac{i}{(4\pi i t)^{n/2}}
\exp i\left[\frac{|\xx-\yy|^2}{4t} + \frac{x_n+y_n}{2} t - \frac{1}{12}t^3
\right]
~,\qquad t>0 ~.
\label{TDSEfs}
\ee
This is the
fundamental solution to the time-dependent Schr\"odinger equation \eqref{tdse}.
An inverse Fourier transform in time returns to the frequency-domain,
giving the desired \eqref{Phiint}.
\end{proof}

\begin{rmk}[plain Helmholtz equation]
\label{r:nograv}
Applying the above technique to the constant-wavenumber Helmholtz equation
$(\Delta + E)u = 0$ gives the fundamental solution representation
$$
\Phi(\xx,\yy) = \frac{i}{(4\pi i)^{n/2}}
\int_0^\infty \frac{1}{t^{n/2}} \exp i\left[\frac{|\xx-\yy|^2}{4t} + E t\right]
dt ~,
$$
which is the same as \eqref{Phiint} absent two terms.
In the case $n=2$, by changing variable
to $s = (2i\sqrt{E}/r)t$, where $r = |\xx-\yy|$, we see that the above is the
little-known
Schl\"afli integral representation \cite[(4) Sec.~6.21]{watson} for the
radiative fundamental solution $(i/4) H^{(1)}_0(\sqrt{E}r)$,
where $H_0^{(1)}$ is the outgoing Hankel function of order zero.
% give Schlafli 
\end{rmk}

\begin{rmk}
\label{r:exist}
We leave for future work a proof that
the causal fundamental solution \eqref{Phiint}
satisifies our radiation condition in Definition~\ref{d:rbcs}, although
physical intuition, the Helmholtz case, and
numerical evidence strongly suggest that this is the case.
A proof seems to demand stationary phase estimates beyond the
scope of this work. %largely numerical paper.
%The proof would also give 
A rigorous existence proof for the BVP \eqref{lhelm}--\eqref{ourk}
would follow,
in an analogous fashion to \cite[Thm.~3.9]{coltonkress}.
\end{rmk}

The importance of \eqref{Phiint} is that
quadrature of this integral will provide us with an 
accurate numerical algorithm to evaluate the fundamental solution for $n=2$
(Section \ref{s:eval}).

Finally we recall a property of $\Phi$ special to $n=2$.
Since the PDE has coefficients which vary as analytic functions of
$x_1$ and $x_2$, the fundamental solution must have the form
\cite[Ch.~5]{garab}
\be
\Phi(\xx,\yy) = A(\xx,\yy) \frac{1}{2\pi} \log\frac{1}{|\xx-\yy|} + B(\xx,\yy)
~,
\label{spike}
\ee
where $A$ and $B$ are analytic in both coordinates of both variables,
and $A(\xx,\xx) = 1$
for all $\xx\in\R^2$. Thus, as with the Laplace and Helmholtz equations,
there is a (positive sign) logarithmic singularity at the source point.

\subsection{Connection to ray dynamics, propagating and forbidden regions}
\label{s:ray}

In Fig.~\ref{f:fs} we plot the fundamental solution, showing the
different behaviors resulting by varying the height of the
source location $\yy$ at fixed energy $E$.
In panel (a) the radiation from the source point is visible, as is
interference between upwards and downwards
propagating waves.
In panel (b) the source is closer to the turning height $x_2=-E$,
and ray trajectories have been superimposed showing the
connection to classical dynamics.
We now review this connection
(see e.g.\ \cite[Sec.~9-10]{kellerrev},
%\cite{babich},
\cite[Sec.~4.5]{evanspde},
\cite[Sec.~5.1]{chapman} \cite{tannorbook}).
Consider the general variable-coefficient Helmholtz equation
$$(\Delta + k(\xx)^2) u = 0
~.
$$
When $k$ is locally large, inserting Keller's traveling wave
ansatz $u(\xx) = a(\xx) e^{i \phi(\xx)}$ into the PDE
gives to leading order the eikonal equation $|\nabla \phi| = k(\xx)$,
% should I really do the t-dep case as in Evans Ch 4.5?
whose characteristics are rays given by evolving Hamilton's
equations (here a dot %$\cdot$
indicates a time derivative),
\be
\dot{\xx} = \nabla_\pp H~,
\qquad \dot{\pp} = -\nabla_\xx H~,
\label{hamil}
\ee
with the Hamiltonian $H(\xx,\pp) = |\pp|^2 + V(\xx)$
and potential $V(\xx) = -k(\xx)^2+E$, with (conserved) total energy $H=E$,
where $E$ is any constant.
(Here the kinetic energy term corresponds to a particle of mass $\half$.)
Another way to express this is via quantization,
or ``quantum-classical correspondence'',
which associates the operator $i\nabla$ with the momentum variable $\pp$.
This rigorous connection is the topic of semiclassical analysis \cite{zworski}.

Returning to our case of stratified $k$, and the constant $E$,
given by \eqref{ourk}, then $V(\xx) = -x_2$,
we see that rays evolve under a constant ``gravitational''
force field $-\nabla V(\xx) = (0,1)$
in the vertical direction,
i.e.\ Hamilton's equations are $\dot{\xx} = 2\pp$ and
$\dot{\pp} = (0,1)$. To model the fundamental solution $\Phi(\cdot,\yy)$, rays
are launched from the source $\yy$, with initial momentum ${\bm\rho}=(\rho_1,\rho_2)$,
hence have the Galilean solution
\be
x_1(t) = y_1 + 2\rho_1t~, \qquad
x_2(t) = y_2 + 2\rho_2t + t^2~.
%                                    call these q ? as in class mech?
\label{galileo}
\ee
Fig.~\ref{f:fs} suggests that such rays
predict the wavefronts and caustics of $\Phi$, and that $\Phi$ is
small in the ``classically forbidden'' region, which we call region F,
defined in the following.
\begin{prop} % pppppppppppppppppppppppppppp
Rays obeying \eqref{hamil} with
Hamiltonian $H(\xx,\pp) = |\pp|^2 - x_2$
launched from $\yy$ with total energy $E$ cannot reach the forbidden
region F, which is defined by $\xx=(x_1,x_2)$ such that
\be
\frac{|\xx-\yy|}{2} \; > \; \frac{x_2 + y_2}{2} + E~,
\label{forbid}
\ee
whose boundary is the parabola
with focus $\yy$ and directrix $x_2 = -y_2-2E$.
Rays can reach any point in the complement of the region,
which we will label region A, for ``classically allowed''.
\label{p:rays}
\end{prop}  % ppppppppppppppppppppppppppp
We provide a proof, simplifying that of
Bracher et al.~\cite{bracher},
that introduces the concept of {\em travel time},
crucial to the later numerical evaluation.
\begin{proof}
We substitute the formulae for $\rho_1$ and $\rho_2$ from
\eqref{galileo} into the expression $\rho_1^2+\rho_2^2 = E + y_2$
expressing that the  initial total energy $H(\yy,\bm\rho)=E$,
to get the quadratic equation in $t^2$,
\be
\frac{t^4}{4} - b t^2 + a = 0
\label{tquad}
\ee
where for later simplicity we define
\be a := \frac{|\xx-\yy|^2}{4}, \qquad
b := \frac{x_2+y_2}{2}+E
~.
\label{ab}
\ee
The positive solutions to \eqref{tquad}
give possible ray travel times from $\yy$ to $\xx$ at
fixed $E$, being
\be
t_{\pm} = +\sqrt{2 \bigl( b \pm \sqrt{b^2-a} \bigr) }
~.
\label{tpm}
\ee
No real solutions are possible precisely when
$\sqrt{a}>b$, which gives \eqref{forbid}.
The boundary, written $2\sqrt{a} = 2b$, states
that the distance from $\yy$ to $\xx$ equals the distance from $\xx$
to the directrix line $x_2 = -y_2-2E$, defining a parabola.
\end{proof}
At the parabolic boundary the two travel times coalesce, i.e.\ $t_- = t_+$,
causing a caustic (singularity in density) for the rays,
which manifests itself as large amplitudes in the fundamental solution;
see Fig.~\ref{f:fs}(a)--(b).
We show in Fig.~\ref{f:fs}(c)
a case where the source itself lies in the forbidden region.
Here there are no classical rays and
the wave leakage into the propagating region is exponentially small,
occurring only in a single upwards direction.

Finally we emphasize that time evolution appears in
two different settings in this section:
in the time-dependent Schr\"odinger equation
to give $t$ in the integral \eqref{Phiint}, and
the time variable $t$ in the classical dynamics.
We have chosen the dimensionless units (i.e.\ particle mass $\half$)
so that they correspond.

% BBBBBBBBBBBBBBBBBBBBBBBBBBBBBBBBBBBBBBBBBBBBBBBBBBBBBBBBBBBBBBBBBBBBBBBBBBBB
\section{Conversion to a boundary integral equation, and its numerical solution}
\label{s:bie}

We will reformulate the exterior Dirichlet BVP \eqref{lhelm}--\eqref{ourk}
as a Fredholm second-kind integral equation on $\pO$.
Since it provides us a useful numerical test case, we also do the same
for the interior BVP.
Recall that by standard elliptic PDE theory,
given a compact domain $\Omega$, the interior
Dirichlet BVP has a unique solution for
all $E$ except at a countable set (the Dirichlet eigenvalues of the operator
$-\Delta -x_2$) that accumulates only at infinity
\cite[Thms.~4.10, 4.12]{mclean}.
% fredholm first kind/second kind integral equations

Given the fundamental solution $\Phi({\bf x}, {\bf y})$, and
a ``density'' function $\tau$ on the boundary curve $\pO$,
we define the standard single- and double-layer potential representations,
\begin{equation}
(\Sing \tau)({\bf x}):= \int_{\partial \Omega} \Phi({\bf x},{\bf y}) \tau({\bf y}) \, ds_\yy
\qquad
(\Doub \tau)({\bf x}):= \int_{\partial \Omega} \frac{\partial \Phi({\bf x},{\bf y})}{\partial {\bf n}_\yy} \tau({\bf y}) \, ds_\yy
~,
\label{reps}
\end{equation}
where $\nn(\yy)$ is the outward-pointing unit normal vector at the point
$\yy\in\pO$, and $ds$ the usual arc length element.
One may interpret $\yy$ as a source point and $\xx$ as a target.
Since limits of such potentials on the curve itself may depend on from
which side it is approached, we define
$$
v^\pm({\bf x}) := \lim_{h\to 0^+} v({\bf x} \pm h {\bf n}({\bf x}))
~.
$$
Letting $S:C(\pO)\to C(\pO)$ be the boundary integral operator with
kernel $\Phi(\xx,\yy)$, and $D:C(\pO)\to C(\pO)$ be the
boundary integral operator with kernel
$\partial\Phi(\xx,\yy)/\partial \nn(\yy)$ taken in the principal value sense,
we have jump relations,
\bea
(\Doub \tau)^\pm (\xx) & = &(D\tau \pm \half\tau)(\xx)
~,
\label{jr1}
\\
(\Sing \tau)^\pm (\xx) & = & (S\tau)(\xx)
~,
\label{jr2}
\eea
which are identical to the Laplace and Helmholtz cases
\cite[Thm.~3.1 and p.66]{coltonkress}.
For the proof we need the variable-coefficient elliptic PDE
case \cite[Thm.~6.11 and (7.5)]{mclean}.

The indirect BIE is constructed by making the ``combined field integral
equation'' (CFIE) ansatz
\be
u = (\Doub - i\eta\Sing)\tau
\label{urep}
\ee
and substituting this into the boundary condition \eqref{bc},
using the exterior jump relations to get the BIE for the unknown density $\tau$,
\be
(\half I + D - i\eta S)\tau = f \hspace{2in} \mbox{exterior BIE}
~,
\label{bie}
\ee
where $I$ is the identity.
This mixture of double- and single-layer prevents
a spurious resonance problem
(for $\eta=0$ the operator would be singular at interior Neumann eigenvalues),
making the BIE a robust method for the BVP.
The choice of constant $\eta$ is not crucial but is commonly
scaled with the wavenumber \cite{kress91}; our wavenumber
varies in space, and we choose at typical value $\eta = \sqrt{E}$.
Note that the correct sign of $\eta$ is crucial for
rapid convergence of iterative solvers at high frequency.

For the interior BVP, the CFIE is not (usually) needed, so we
set $\eta=0$ and get
\be
(-\half I + D)\tau = f \hspace{2in} \mbox{interior BIE}
~.
\label{bieint}
\ee

Note that the operator $S$ is compact, and
when $\pO$ is smooth the operator $D$ is compact, making the above
BIEs of Fredholm second kind.
This has the well-known advantages over first-kind BIEs
of stability under discretization, and
a benign spectrum leading to rapid convergence for the iterative solution of
the resulting linear system.

\subsection{Numerical solution: Nystr\"om method and quadrature}
\label{s:nyst}

We first parametrize the smooth closed curve $\pO$ by a
$2\pi$-periodic function $\zz:[0,2\pi)\to\R^2$ such that $\zz(t) \in \pO$
and $|\zz'(t)| \neq 0$, for all $t \in \R$.
Changing variable to the parameter $t$ turns \eqref{bie} into
a integral equation on the periodic interval $[0,2\pi)$,
\be
\half \tau(t) + \int_0^{2\pi} \left(
\frac{\partial \Phi(\zz(t),\zz(s))}{\partial \nn_{\zz(s)}}
- i\eta \Phi(\zz(t),\zz(s))
\right)|\zz'(s)|\,\tau(s) ds \;=\; f(t), \quad \forall t \in [0,2\pi)
\ee
The reparametrization of \eqref{bieint} is similar.
We can write both of these integral equations in the standard form
\be
\tau(t) + \int_0^{2\pi} K(t,s)\tau(s)ds  = g(t),  \quad \forall t \in [0,2\pi)
\label{bieK}
\ee
In the exterior case, we see from the presence of $\Phi$ and from
\eqref{spike} that $K$ has a
logarithmically singular kernel, i.e.\ $K(s,t) \sim \log|s-t|$;
in the interior case the kernel of
$K$ is continuous at the diagonal but has a weaker singularity
of the form $|s-t|^2 \log |s-t|$, as with the Helmholtz
equation \cite[Sec.~3.5]{coltonkress}.
To achieve high-order convergence in either case
when the data $g$ is smooth we will need to
use a quadrature scheme accurate for kernels containing a periodized
log singularity of the form
\be
K(t,s) = K_1(t,s) \log \left(4 \sin^2 \frac{s-t}{2} \right) + K_2(t,s)
\label{split}
\ee
where $K_1$ and $K_2$ are smooth and $2\pi$-periodic in both of their
arguments.

We apply the Nystr\"om method \cite[Sec.~12.3]{na}
to approximate the solution of \eqref{bieK}
by that of a linear system, based upon an underlying quadrature rule.
For this we use periodic trapezoid rule quadrature,
\be
\int_0^{2\pi} \phi(t) dt \approx \frac{2\pi}{N} \sum_{j=1}^N \phi(s_j),
\qquad \mbox{where } s_j = 2\pi j/N
\label{ptr}
\ee
whose approximation error for a $2\pi$-periodic $\phi\in C^\infty(\R)$
is super-algebraic, i.e.\ $O(N^{-m})$ for each $m>0$ \cite[Cor.~9.27]{na}.
The first step in the Nystr\"om method is to enforce
\eqref{bieK} only at the nodes $\{s_i\}$, giving
\be
\tau(s_i) + \int_0^{2\pi} K(s_i,s)\tau(s)ds  = g(s_i),  \quad \forall i = 1,\ldots, N
\label{seminyst}
\ee
Were $K$ to possess a smooth kernel (i.e.\ $K_1\equiv 0$),
superalgebraic convergence
would be achieved by applying \eqref{ptr} to the above integral,
to give the square $N$-by-$N$ linear system,
\be
\tau_i + \sum_{j=1}^N A_{ij} \tau_j = g_i, \quad \forall i = 1,\ldots, N
\label{linsys}
\ee
with elements of the matrix given by
\be
A_{ij} = \frac{2\pi}{N} K(s_i, s_j)
~,
\label{Asmooth}
\ee
and where $\tau_j$ approximates $\tau(s_j)$ and the right-hand side
vector has elements $g_j=g(s_j)$.

However, for general singular kernels of the form \eqref{split}, the formula
\eqref{Asmooth} fails to be accurate, and diagonal entries would be infinite.
Yet it is still possible to design a set of quadrature nodes to approximate
the integral in
\eqref{seminyst} to high accuracy for kernels of the form \eqref{split}.
This is done by replacing a few of the trapezoid nodes $s_j$
near the singularity $s_i$ by
a new set of auxiliary nodes and weights;
we choose 16th-order Alpert end-correction
nodes \cite{alpert}, of which 30 are required (15 either side of the singularity).
The auxiliary node nearest the target point is at a distance
of around $10^{-3} \delta$ from this target point, where
$\delta \approx (2\pi/N)|\zz'(s_i)|$ is the local underlying node spacing.
The values of $\tau$ at these auxiliary nodes is related to the
neighboring few elements
of the vector $\{\tau_j\}_{j=1}^N$ using local Lagrange interpolation.
The net effect is that the matrix $A$ takes the form \eqref{Asmooth}
away from the diagonal, but with corrected entries near the diagonal.
The full formulae are presented in \cite[Sec.~4]{hao}.
This gives for kernels of the form \eqref{split} a high-order convergence
of the error between $\tau_j$ and the true solution samples $\tau(s_j)$
of $O(N^{-16} \log N)$,
for either the exterior or interior BIEs of interest.
For the convergence theory see
\cite[Cor.~3.8]{alpert} for the end-correction scheme,
and Kress \cite[Ch.~12]{LIE}.

Once the linear system \eqref{linsys} has
been solved, the vector ${\bf \tau} = \{\tau_j\}_{j=1}^N$
may be used to reconstruct the scattered potential at any target location
sufficiently far from $\pO$, by substituting the same trapezoid rule into
the integrals \eqref{reps} in
the representation \eqref{urep}, to get
\be
u(\xx) \; = \; \sum_{j=1}^N \left(
\frac{\partial \Phi(\xx,\zz(s_j))}{\partial \nn_{\zz(s_j)}}
- i\eta \Phi(\xx,\zz(s_j))
\right)|\zz'(s_j)|\,\tau_j
\label{ueval}
\ee
A rule of thumb is that this quadrature rule is accurate for all
points at least $5\delta$ from the
boundary
%, where $h$ is the local node spacing
\cite[Remark~6]{ce}.
%\eqref{ueval}
%applies to the exterior case;
As before, for the interior case we set $\eta=0$.

\bfi % ffffffffffffffffffffffffffffffffffffffffffffffffffffffffffffffffffff
\centering\ig{width=5in}{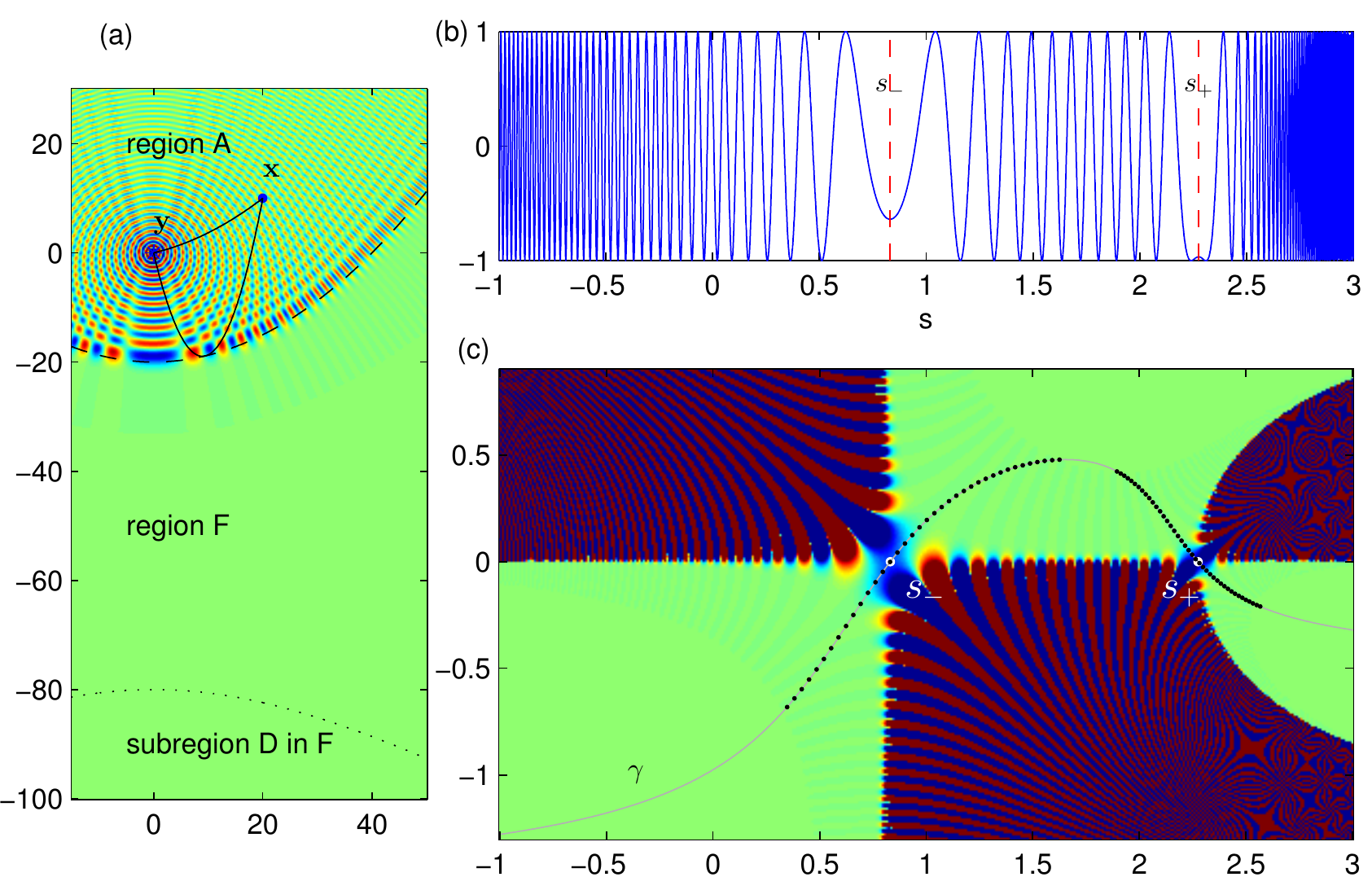}
\caption{
Contour integration for the fundamental solution at $E=20$, $\yy=\mbf{0}$.
(a) $\re \Phi(\xx,\yy)$ in the physical domain of $\xx$,
showing the two ray paths to reach $\xx = (20,10)$,
and the classically allowed (A), forbidden (F), and deep forbidden (D) regions.
(b) integrand of \eqref{Phis} on the real $s$ axis,
with the two stationary phase points $s_\pm = \log t_\pm$.
(c) real part of the same integrand in the complex $s$ plane,
saddle points (white dots), and
the 79 quadrature nodes used lying on the contour $\gamma$.
In (a) and (c) the color scale is
blue (negative) through green (zero) to red (positive); in (c) the
color range covers $[-1,1]$.
\label{f:integA}}
\efi

% EEEEEEEEEEEEEEEEEEEEEEEEEEEEEEEEEEEEEEEEEEEEEEEEEEEEEEEEEEEEEEEEEEEEEEEEEEEEE
\section{Evaluation of the fundamental solution}
\label{s:eval}

Filling the Nystr\"om matrix $A$ of the previous section,
and evaluating the solution $u$ via \eqref{ueval},
both demand a large number of evaluations of $\Phi(\xx,\yy)$,
from source points $\yy$ that are either periodic trapezoid nodes $\zz(s_j)$ or
auxiliary nodes. When filling $A$ the target points $\xx$ are also
the nodes $\zz(s_i)$, thus for a small number of
cases ($O(N)$ of them), the distance $|\xx-\yy|$ will be very small
(e.g.\ $10^{-3} \delta$).

As promised, we base our evaluation of the fundamental solution
on the $n=2$ dimensional case of \eqref{Phiint},
\be
\Phi(\xx,\yy) \;=\;
\frac{1}{4\pi}
\int_{0}^\infty \frac{1}{t} \exp i\left[\frac{|\xx-\yy|^2}{4t} +
\left( \frac{x_n+y_n}{2} + E\right) t - \frac{1}{12} t^3
\right]
ds
\;=\;
\frac{1}{4\pi}
\int_{0}^\infty \frac{1}{t} e^{i \psi(t)} dt
\ee
where, recalling \eqref{ab},
the {\em phase function} $\psi(t) = \psi_{a,b}(t)$ is defined by
\be
\psi(t) := \frac{a}{t} + bt - \frac{1}{12} t^3
~.
\label{psi}
\ee
To remove the pole at the origin, and place small and large $t$
on an equal footing, we change variable via $t = e^s$ to get
\be
\Phi(\xx,\yy) \;=\;
\frac{1}{4\pi} \int_{-\infty}^\infty \exp i \psi(e^s) \, ds
~.
\label{Phis}
\ee
This integrand is shown in \fref{f:integA}(b), for $E=20$ and the
source $\yy$ and target $\xx$ shown in \fref{f:integA}(a).
It is clearly highly oscillatory---and it becomes more so with increasing
$E$---thus accurate integration along
the real $s$ axis would be prohibitively expensive.
However, $\phi(e^s)$,
%$\mathbb{C} \backslash \{z: \re z = 0, \im z \le 0\}$, thus
and hence the integrand, is analytic in the entire
complex $s$ plane.
We thus use numerical saddle point integration
\cite[Sec.~5.5]{temme} \cite{borncont}
(related to, but simpler than, ``numerical steepest descent'' \cite{huy}),
along a contour passing through the stationary phase (saddle) points
and asymptotically tending to the correct regions of the plane.
We have the following by direct differentiation of \eqref{psi}.
\begin{prop} % pppppppppppppppppppppppppppppppppppppppp
Given a source $\yy$, target $\xx$, and energy $E$, 
the stationary phase points, that is, the solutions to $\psi'(t) = 0$,
are precisely the classical ray travel times $t_\pm$ already given by
\eqref{ab}-\eqref{tpm}.
\label{p:saddle}
\end{prop}    % pppppppppppppppppppppppppppppppppppp
This connection between waves and rays is key to our
efficient numerical evaluation of the integral \eqref{Phis}.

\begin{rmk} % rrrrrrrrrrrrrrrrrrrrrrrr
%Although it does not have numerical consequences,
There is a beautiful and
deep physical reason lying behind Prop.~\ref{p:saddle}, i.e.\ $\psi'(t_\pm)=0$.
The phase function (term in square brackets)
in the time-dependent Schr\"odinger
propagator \eqref{TDSEfs} is the classical {\em action} $S(\xx,\yy;t)$,
defined as the time integral over $[0,t]$
of the Lagrangian along the unique classical
path from $\yy$ to $\xx$ taking precisely time $t$ \cite[Sec.~2]{bracher}
\cite[Ch.~10]{tannorbook}.
(Note that for a general potential function $V(\xx)$, this is only approximately
true in the semi-classical or high-frequency limit; its exactness
here reflects exact formulae for the propagation of the Gaussian
when the potential is at most quadratic in the coordinates
\cite{hellerhouches}.)

Inserting this into the last step in the proof of Lemma~\ref{l:fs},
we see that $\Phi(\xx,\yy) = \int_0^\infty \exp i \left[S(\xx,\yy;t)
+ Et\right] \, dt$,
thus the phase function \eqref{psi} is $\psi(t) = S(\xx,\yy;t) + Et$.
A less well-known result from classical mechanics is
$\partial S(\xx,\yy;t)/\partial t |_{\xx,\yy} =
-E_{\xx,\yy}(t)$,
where
$E_{\xx,\yy}(t)$ is the energy
required to complete the path in time $t$.
\cite[Ex.~10.4(c)]{tannorbook}.
Thus $\psi'(t)=0$ precisely when $E_{\xx,\yy}(t)=E$, that is, at the travel times
for a ray at the particular energy $E$ to pass from $\yy$ to $\xx$.
\label{r:phase}
\end{rmk} % rrrrrrrrrrrrrrrrrrrrrrrrrrrrr
% This took me ages to understand

An example contour passing through the (real-valued) saddle points
and ending in the correct regions of the plane is shown in
\fref{f:integA}(c).
On such a contour the integral may be approximated to exponential
accuracy using the trapezoid rule \cite{PTRtref}
(with respect to the variable parametrizing the contour),
and the sum may be truncated once values are sufficiently small.

\bfi % fffffffffffffffffffffffffffffffffffffffffffffffffffffffffffffffffffff
\ig{width=3.2in}{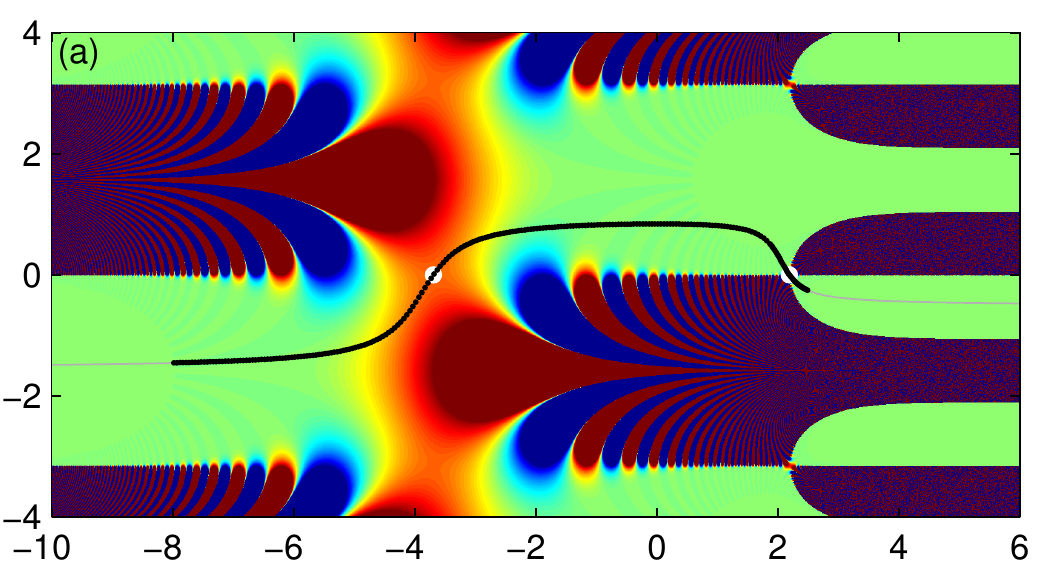}
\qquad
\ig{width=1.7in}{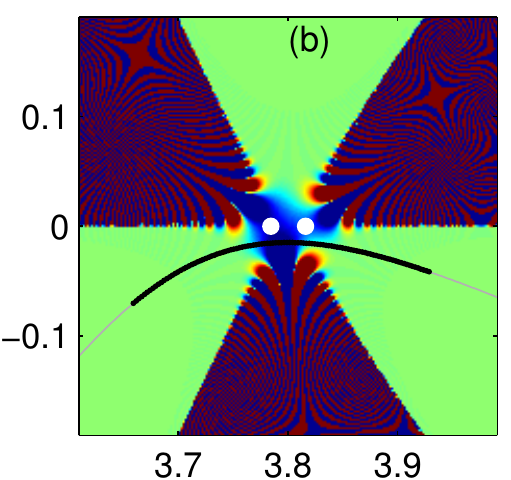}
\\
\ig{width=3.2in}{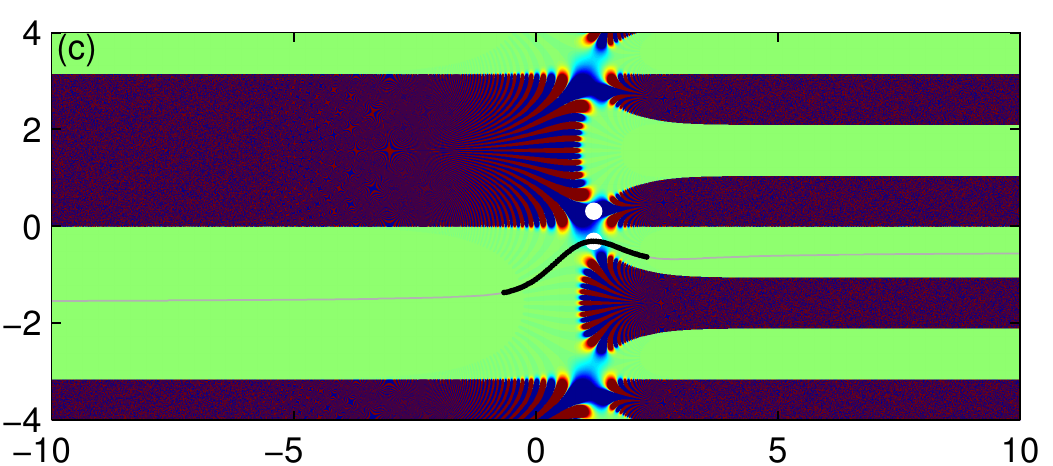}
\qquad
\ig{width=1.7in}{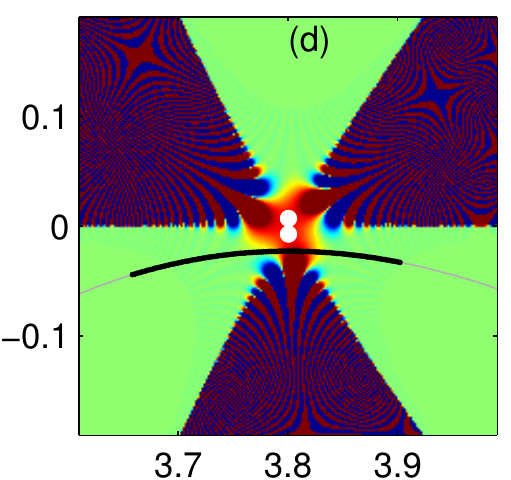}
\\
\ig{width=3.2in}{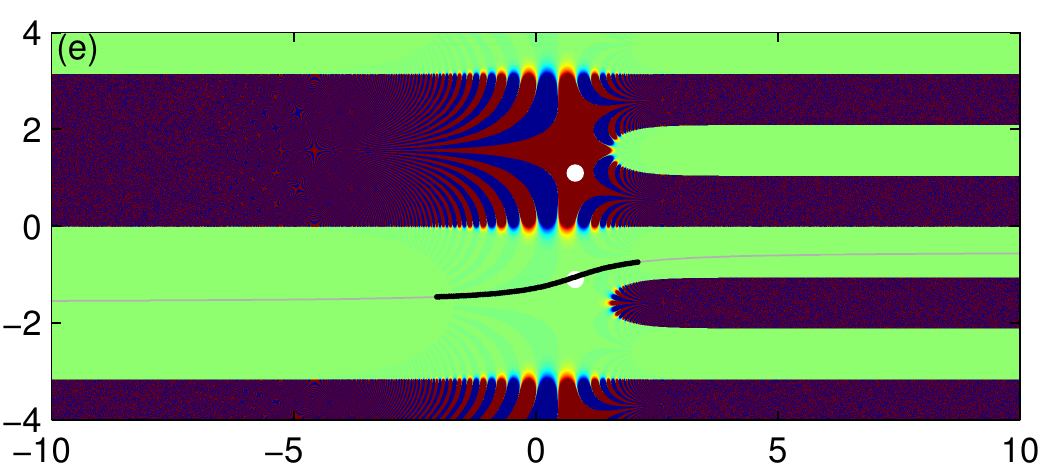}
\ig{width=3.2in}{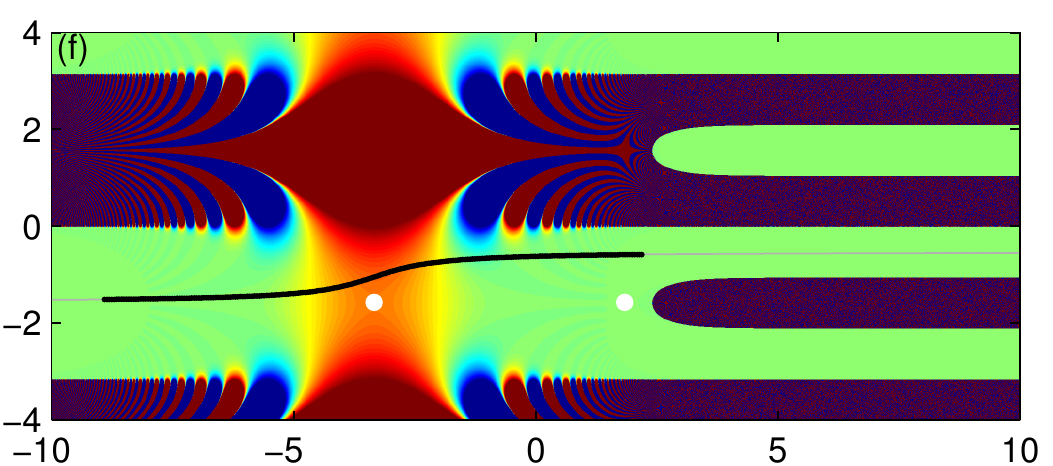}
\caption{
Real part of integrands plotted in the complex $s$ plane, for source
$\yy=\mbf{0}$, with saddle points (white dots) and numerical
integration contours (grey) and nodes (black).
(a) Region A (allowed) but source close to target, $E=20$, $\xx=(0.1,0.2)$.
(c) Region F (forbidden), $E=10$, $\xx=(1,-11)$.
(b) and (d) Zoom in on coalescing saddle points:
at $E=10^3$, with $\xx=(2E-1,0)$ in (b) (just allowed),
and $\xx=(2E+0.2,0)$ in (d) (just forbidden).
(e) Region D (deep forbidden), $E=1$, $\xx=(1,-5)$.
(f) Region D but source close to target, $E=-10$, $\xx=(0.1,0.2)$.
\label{f:integmulti}
}
\efi

\subsection{Choice of saddle point contour}

Since the integrand in \eqref{Phis} is entire, mathematically
the choice of contour is irrelevant
as long as its ends connect $-\infty$ to $+\infty$.
However, for practical numerical evaluation the contour choice is crucial.
Observe in \fref{f:integA}(c) that the integrand is exponentially small
in some regions, exponentially large in others, and that
the borders between them are quite well defined.
One may deform the limits of the contour to lie below the real axis,
as long as one stays within the exponentially small regions
adjoining the real axis (lower-left and lower-right in \fref{f:integA}(c)).
It must connect these limits, but
to prevent catastrophic cancellation it must avoid large regions,
passing between small regions only via saddle points,
and passing through these saddle points at an angle not too
far from the steepest descent direction.
In addition, an analytic contour shape is desirable, since the trapezoid
rule is then exponentially convergent.
See \fref{f:integA}(c) and \fref{f:integmulti} for examples.

The task remains to choose, for any parameters $a$ and $b$,
a good contour, and rules for choosing the trapezoid node spacing
and truncation intervals. Our rules will depend on the existence
and types of classical rays.
%Given \pref{p:saddle}, it is not surprising that the existence and
%type of rays passing from $\yy$ to $\xx$ at energy $E$ has a big impact
%on the good contour shapes needed for integrand \eqref{Phis}.
Recall the definition that the set $\xx$, $\yy$ and $E$ is
classically allowed (region A) if there is one or two rays connecting $\yy$ to
$\xx$ at energy $E$ in (real-valued) time, otherwise forbidden (region F).

\subsubsection{Classically allowed (region A): $b^2\ge a$}

In this case, as in \fref{f:integA}, there are two real saddle points,
with steepest descent angles $\pi/4$ for $s_-=\log t_-$
(the root with smaller real part), and $-\pi/4$ for $s_+=\log t_+$.
We parametrize contours by their real part $\alpha\in\R$,
thus
$$
s = \gamma(\alpha) := \alpha + ig(\alpha),
\qquad \mbox{ hence } \quad \gamma'(\al) = 1 + ig'(\al),
$$
where the function $g:\R\to\R$ depends on the usual parameters $a$ and $b$
\eqref{ab}.
The following analytic function $g$ makes the contour
pass through the two saddle points at angles not too far from $\pm\pi/4$,
\be
g(\alpha) = \bigg[\biggl(\frac{1}{\pi} + \frac{1}{2}\biggr)
\tan^{-1} \bigl(2(\al - \re s_-+c_-)\bigr)
- \biggl(\frac{\pi}{4} - \frac{1}{2}\bigg)
\bigg] \cdot \bigg[
\biggl(\frac{1}{\pi} + \frac{1}{6}\biggr)
\tan^{-1}\bigl(-4(\al - \re s_+ -c_+)\bigr) -
\biggl(\frac{\pi}{12} - \frac{1}{2}\biggr) \bigg]
\label{gA}
\ee
with the constants
$c_- := \half \tan(\frac{\pi^2 - 2\pi}{4 +2\pi})$
and
$c_+ := \sfrac{1}{4}\tan(\frac{\pi^2 - 6\pi}{12 +2\pi})$.
We do not claim it is optimal, but it serves our purpose well.
Examples from this family
are shown in \fref{f:integA}(c) and \fref{f:integmulti}(a).

The leftward limit $\lim_{\alpha\to-\infty}g(\alpha) = -\pi/2$ is
designed to lie in the middle of the exponentially-small region to the
left. This region has height $\pi$ due to the
$2\pi$ vertically periodic nature of the
function $e^{-s}$ which dominates as $\re s$ becomes highly negative.
To the right the period becomes three times smaller, since $e^{3s}$ is
dominant, thus we chose $\lim_{\alpha\to\infty}g(\alpha) = -\pi/6$.
Note that it is essential to enter and exit through the correct
periodic images on the left and right sides.

When the saddle points coalesce ($t_-=t_+$ at the classical
turning point, or boundary of A and F),
the angles through the saddle points become flatter, as is needed to
traverse smoothly through the small region;
see the zoom \fref{f:integmulti}(b).
However, when saddles are close to coalescing
at high $E$,
it is advantageous for accuracy to shift the contour
down enough to avoid being close to the rapid oscillations on the
real axis, whilst keeping the integrand not too large.
Hence, when $|s_+-s_-|<0.1$ we add the constant
\be
c_\tbox{shift} := -i \, \min \left[\frac{0.7}{\sqrt{E}}, 0.1 \right]
\label{shift}
\ee
to $\gamma$. The resulting shift is visible in the figure.

\subsubsection{Classically forbidden (region F): $b^2 < a$}

Things get simpler when no real rays are possible: the
saddle points $s_\pm$ split away from the real $s$ axis,
and only the one with negative imaginary part is relevant.
Let us call this point $s_0$.
There are a couple of regimes to consider;
see \fref{f:integmulti}(c)--(f).
We use the following contour when $\im s_0>-\pi/3$,
$$
g(\al) = \im s_0 +\biggl( \frac{\tan^{-1} (\al - \re s_0) - \pi}{3}
- \im s_0 \biggr)
\bigl(1 - e^{-(\al-\re s_0)^2}\bigr)
~.
$$
This has the same limits as \eqref{gA}, is
designed to pass through $s_0$ horizontally (i.e.\ $g'(\re s_0)=0$),
and is shown in \fref{f:integmulti}(c).
The need for horizontal passage is to stay below the real axis
when saddles are close to coalescing at high $E$.
As above, we also apply the shift \eqref{shift} when saddle points are close.
This is shown in the zoom \fref{f:integmulti}(d).

When $\im s_-<-\pi/3$, we are deep into the forbidden region
(thus we call the region D $\subset$ F).
It lies below the hyperbola $b = -\sqrt{a}/2$ in the $\xx$ plane,
as shown in \fref{f:integA}(a).
In region D we use the simple contour
$$
g(\al) = \frac{\tan^{-1} (\al - \re s_0) - \pi}{3}
~.
$$
This lies above all saddle points, has the same limits
as \eqref{gA}, and is shown in \fref{f:integmulti}(e).

When $b<-\sqrt{a}$, as occurs in region D with negative $E$
and close source-target distances,
the saddle points finally merge again onto the line $\im s = -i\pi/2$.
In this case we take $s_0$ to be the point with more negative real part,
and use the above contour.
This is shown in \fref{f:integmulti}(f).

\bfi  % ffffffffffffffffffffffffffffffffffffffffffffffffffffffffffff
\ig{width=3.2in}{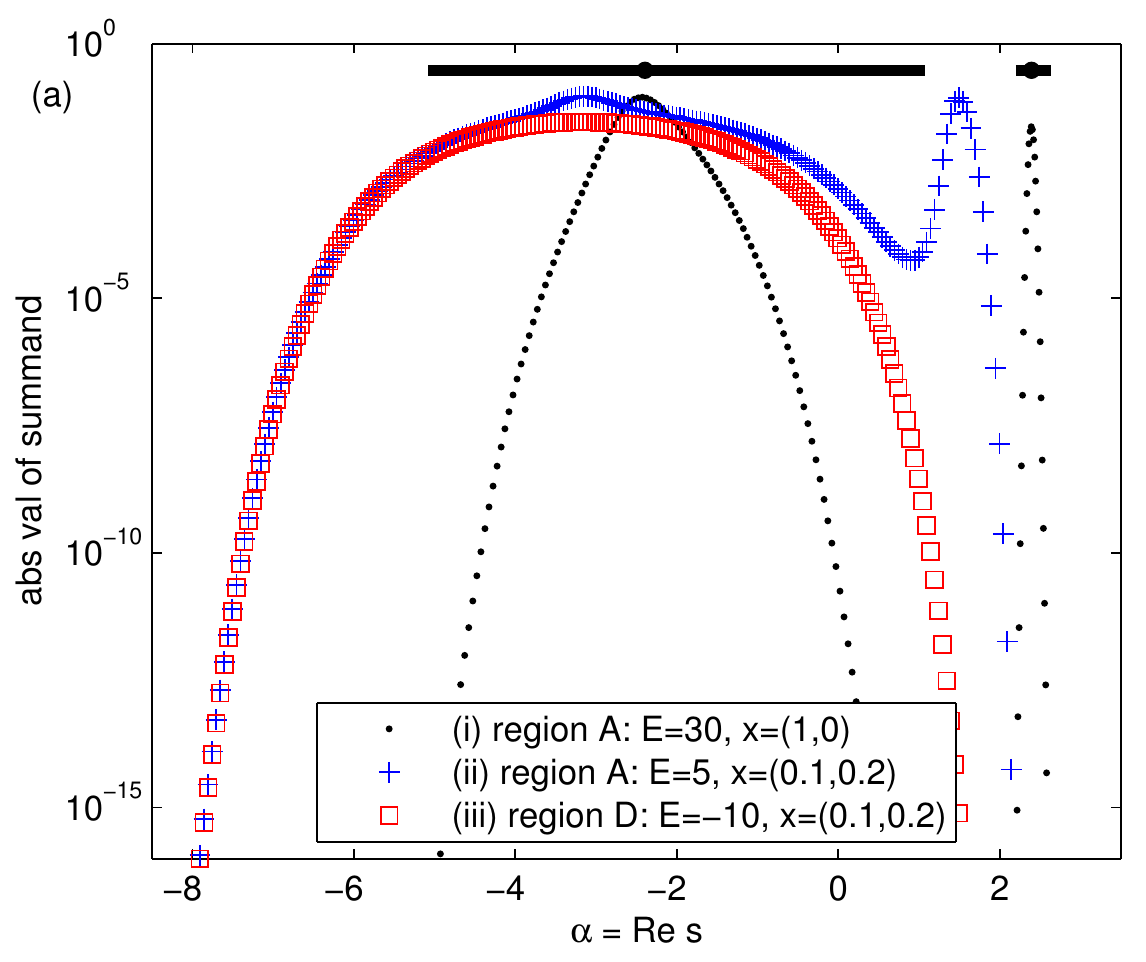}
\quad
\ig{width=3.2in}{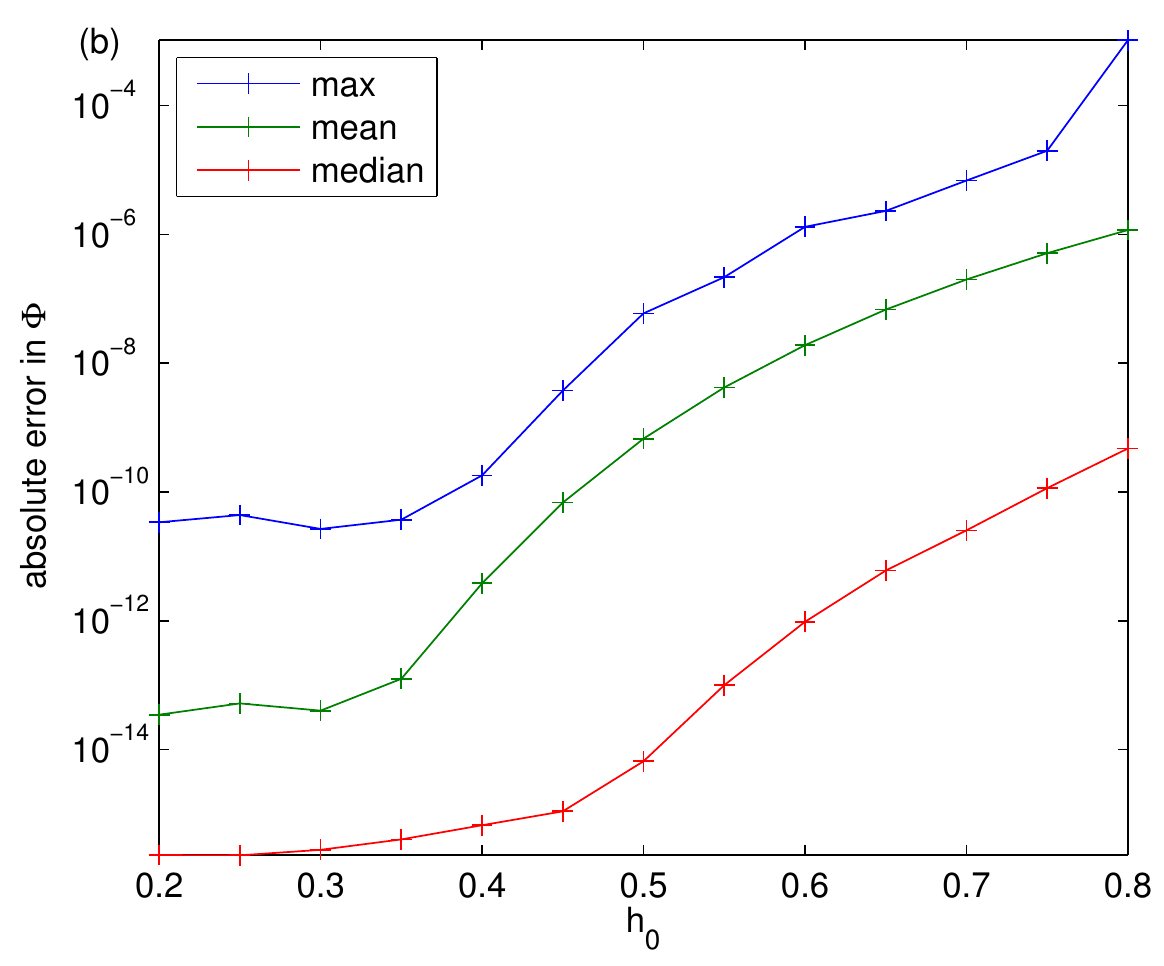}
\caption{
(a) Magnitude of summand in \eqref{TR} along the parametrized contour,
showing three types of behavior.
For case (i) the intervals $I_1$ and $I_2$ containing the saddle
points (large dots) are shown at the top.
(b) Convergence of absolute error in $\Phi$, with respect to
the quadrature spacing $h_0$, also scaling $h_\tbox{max}=0.13\,h_0$ and
$n_\tbox{min} = 15/h_0$.
The source is $\yy=\mbf{0}$, and targets are a set of $10^4$ points
randomly distributed uniformly in angle and uniformly
in the logarithm of distance from the origin, $|\xx|\in[10^{-4},10^4]$.
For each target the set of $E$ tested is
$[-100,-30,-10,-3,-1,1,3,10,30,100,300,10^3,3\times10^3,10^4]$.
The maximum, mean, and median error is taken over the $1.5\times 10^5$
evaluations.
\label{f:bumps+conv}}
\efi

\subsection{Truncation of the integration domain}
\label{s:trunc}

With contour shapes now defined for all cases of $a$ and $b$,
we need rules to truncate the integral to a finite domain
$I\subset\R$, that is,
\be
\Phi(\xx,\yy) =
\frac{1}{4\pi} \int_{-\infty}^\infty \exp i \psi(e^s) \, ds
= \frac{1}{4\pi} \int_{-\infty}^\infty \exp i \psi(e^{\gamma(\al)})\, \gamma'(\al) d\al
\approx
\frac{1}{4\pi} \int_I \exp i \psi(e^{\gamma(\al)})\, \gamma'(\al) d\al
~.
\label{int}
\ee
For efficiency, we wish $I$
%Since the evaluation cost is dominated by evaluation of the complex
%exponentials and arctangents at quadrature nodes, we wish to minimize
%the number of nodes, firstly by choosing $I$
to enclose only
the parts of $\R$ where the integrand is significant, which we define
as exceeding a convergence parameter $\eps$, which we set to $10^{-14}$.
We exploit the fact that, along the contour,
the integrand decays exponentially away from saddle points.

There are three types of behavior:
(i) $I$ comprises two intervals $I_1$ and $I_2$
that may be integrated independently,
(ii) there are two saddle points but the integrand does not die to $\eps$
between them, so it must be handled as a single integration interval,
and (iii) there is one saddle point hence only a single ``bump''
and a single interval.
For case (i), for high $E$ the size of the intervals can be much smaller
than their separation, so integrating them separately is crucial.
All three cases are shown in \fref{f:bumps+conv}(a).
In region A, (i) and (ii) may occur; in regions F and D only (iii) occurs.

The recipe for regions F and D, with one saddle $s_0$,
is to initialize distances $d_1=d_2=|\re s_0|/2$ which define an
interval $[\re s_0-d_1,\re s_0+d_2]$.
If $|\psi(e^{\gamma(\re s_0-d_1)})|>\eps$ then we set $d_1$ to $\beta d_1$,
where $\beta$ is a ``jump factor'' constant, and repeat until the
left end of the interval has integrand no larger than $\eps$.
The same is done for $d_2$ on the right end.
We find that $\beta=1.3$ is a good compromise between making jumps that don't
produce an overly large interval, yet don't require too many extra
integrand evaluations.

The recipe for region A, with saddle points $s_\pm$,
is to use a crude minimization of
$|\psi(e^{\gamma(\al)})|$ in $[\re s_-,\re s_+]$, and if the minimum value exceeds
$\eps$, to use a single interval $[\re s_- - d_1,\re s_+ + d_2]$,
which is initialized and expanded as before.
Otherwise two intervals $I_1$ and $I_2$ are used centered at $s_-$ and $s_+$
respectively, and each is expanded separately, as before.
An example result is shown at the top of \fref{f:bumps+conv}.
%The initial values of $d_1$ and $d_2$ in each case is based upon 
% too much detail.

\subsection{Choice of quadrature node spacing}

For each interval $I'$ ($=I$, $I_1$ or $I_2$),
we need rules to choose $h$, the quadrature node spacing in
the trapezoid rule approximation to \eqref{Phis},
\be
\frac{1}{4\pi} \int_{I'} \exp i \psi(e^{\gamma(\al)})\, \gamma'(\al) d\al
\;\approx\;
\frac{h}{4\pi} \sum_{hj \in I'} \exp i \psi(e^{\gamma(hj)})\, \gamma'(hj)
~.
\label{TR}
\ee
A general rule is to scale $h$ in proportion
to the minimum {\em width} of any saddle points contained in $I'$.
%The intuition here is that, along a steepest descent contour the integrand
%is well approximated by a Gaussian, for which the trapezoid rule
%is (super-)exponentially convergent (the latter can be proved via
%the Poisson summation formula).
Let $s_0$ be such a saddle point, then we define its width
%
%\footnote{This is essentially the standard deviation of the Gaussian
%traversed in the steepest descent direction.}
as
$$
\sigma(s_0) \;:=\; \left|\frac{d^2}{ds^2}\psi(e^s)|_{s=s_0}\right|^{-1/2}
~.
$$
Setting a convergence parameter $h_0$, we use a node spacing of
$$
h \;=\;
\min \left[h_\tbox{max}, \,\frac{|I'|}{n_\tbox{min}}, \,\sigma h_0 \right]  
$$
where $\sigma = \sigma(s_0)$ for the case of one saddle, or
$\sigma = \min[ \sigma(s_-), \sigma(s_+)]$ in the case of two.
The new numerical parameters here are $h_\tbox{max}$, the maximum
allowed node spacing, and $n_\tbox{min}$, the minimum allowed
number of nodes over the interval length $|I'|$.
Both are needed to prevent $h$ from become too large, since
$\sigma$ can be arbitrarily large, e.g.\ when saddles coalesce
or when $|\xx-\yy|$ is very small.
%We find that $h_\tbox{max}=0.05$ and $n_\tbox{min} = 40$ work well.

\bfi % fffffffffffffffffffffffffffffffffffffffffffffffffffffffffff
\quad\ig{width=6in}{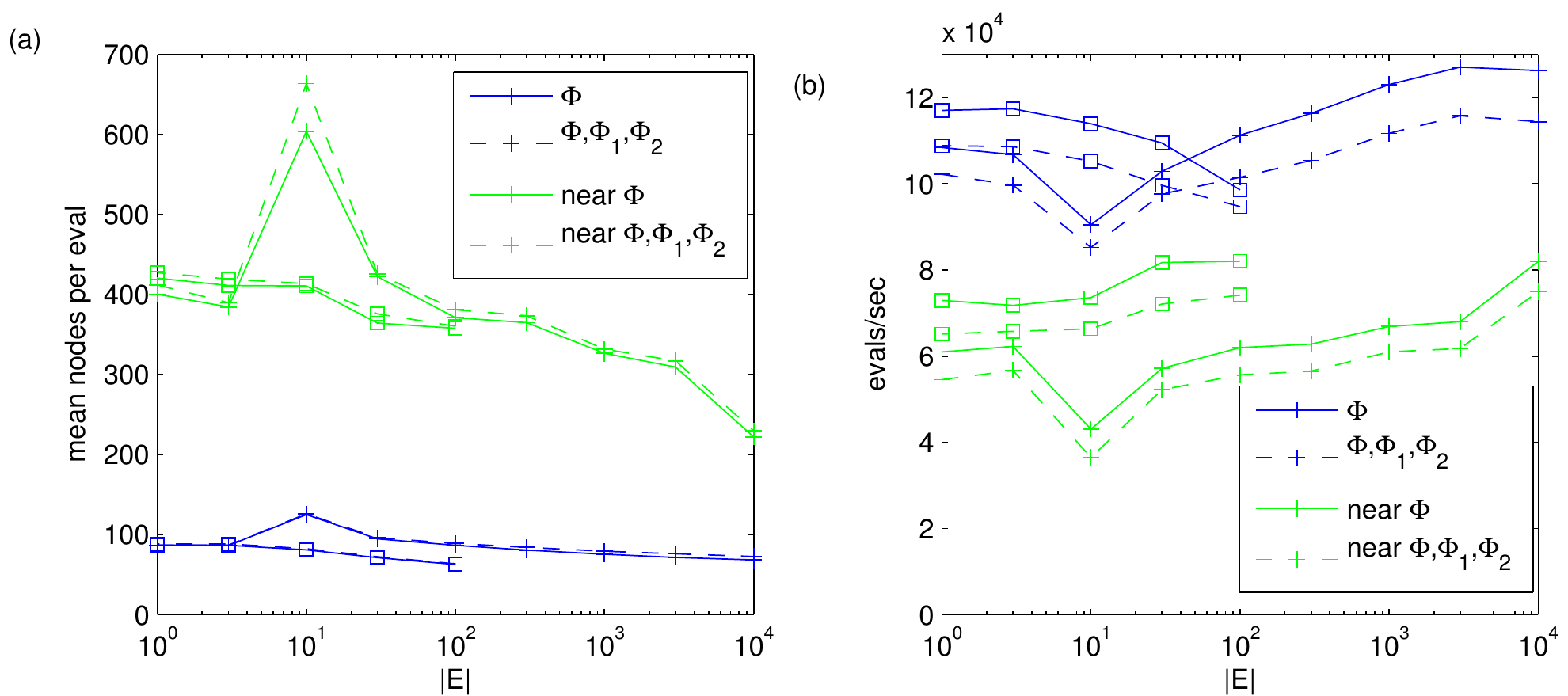}
\caption{
Efficiency of the numerical steepest descent algorithm
as a function of frequency parameter $E$.
(a) Mean number of quadrature nodes used, and (b)
mean number of evaluations per second.
In both graphs,
$+$ signs
indicate $E>0$ while $\square$ signs indicate $E<0$.
Solid lines are for evaluation of $\Phi$ alone while
dashed lines are for evaluation of $\Phi$ and its first partials.
The source is $\yy=\mbf{0}$, and averaging is done over $10^4$
targets randomly distributed uniformly in angle and uniformly
in the logarithm
of distance from the origin.
For the darker (blue) lines $|\xx|\in[10^{-1}, 10^4]$, while
for the lighter (green) lines only ``near'' distances
$|\xx|\in[10^{-4},10^{-1}]$ are used.
\label{f:time}}
\efi

\subsection{Derivatives of $\Phi$}
\label{s:deriv}

The formula for entries of the matrix approximation to the double-layer
operator $D$ in Sections~\ref{s:bie} and \eqref{s:nyst} requires
first derivatives of $\Phi(\xx,\yy)$ with respect to moving the source $\yy$.
These are simple to evaluate from \eqref{psi}--\eqref{Phis}
by passing the derivative through the integral to give,
\bea
\frac{\partial \Phi({\bf x},{\bf y})}{\partial y_1}
&=&
\frac{1}{4\pi}\int_{-\infty}^\infty \frac{ -i (x_1-y_1) e^{-s}}{2}
\exp i \psi(e^s)\,ds
\label{dphi1}
\\
\frac{\partial \Phi({\bf x},{\bf y})}{\partial y_2}
&=&
\frac{1}{4\pi}\int_{-\infty}^\infty \frac{-i\bigl((x_2-y_2)e^{-s} - e^{s}\bigr) }{2}
\exp i \psi(e^s)\,ds
\label{dphi2}
\eea
These may be evaluated with minimal extra effort along with $\Phi$
by including extra factors in \eqref{TR}.
Although these factors can grow exponentially in size, they do not
affect the super-exponential decay away from saddle points of the integrand.
We take care to include these factors when testing for decay of
the integrand to $\eps$ in Sec.~\ref{s:trunc}.

\subsection{Convergence and speed tests}
\label{s:phiconv}

We now test the convergence of the above scheme for $\Phi$ and its
derivatives. For true convergence, $h_0$ must shrink while $h_\tbox{max}$
also shrinks and $n_\tbox{min}$ grows; in \fref{f:bumps+conv}(b) we
perform this test, over the large range of $E$ and $\xx$ parameters
used in \fref{f:bumps+conv}(b), $150000$ in total.
The upper graph shows that a worst-case absolute error around
$3 \times 10^{-11}$ for $h_0=0.35$, $h_\tbox{max}=0.05$ and
$n_\tbox{min}=43$, which we thus find acceptable and fix as our standard choices.
In fact, the lower graphs show that typical accuracies are much better,
being 13 to 15 digits.
%This $h_0$ is smaller than for a Gaussian integral on real axis.
%$\sqrt{2}$ of this is explained by ...
\begin{rmk} % rrrrrrrrrrrrrrrrrrrrrrrrrrrrrrrrrrrrrr
It is known that 24 nodes is sufficient to integrate the Gaussian
via the trapezoid rule to double precision accuracy, e.g.\
\cite[Remark~2]{nufft}.
Nearly twice this is needed to guarantee accuracy
in our setting,
we believe due to distortion around the saddle
from an exactly quadratic phase function, and the overshoot in interval
size due to $\beta$ exceeding 1.
\end{rmk}

Note that we test {\em absolute} not {\em relative} errors in $\Phi$:
we believe that this is what is relevant for solution of BIEs,
and support this claim in the next section.
Since $\Phi$ is exponentially small in the forbidden region, demanding
high relative error would require more effort, and is unnecessary.

%We consider large $E$ up to $10^4$, but also somewhat negative
%$E$ down to $-10^2$ (the PDE is uninteresting for highly negative $E$,
%since all waves are then evanescent).

In \fref{f:time}(a) we test the mean number of nodes $n$ used
for the contour integral over the test set, splitting the data for
near distances $|\xx-\yy|<0.1$, and for $|\xx-\yy|\ge 0.1$.
For the latter, only around 100 nodes is needed, with a slight
decrease at large $E$.
For near distances (hence $a$ is small),
the saddle point $s_-$ moves leftwards, and the 
width of the significant region around it grows
as shown in \fref{f:integmulti}(a) and (f).
We observe that here $n$ grows like $\log 1/|\xx-\yy|$.
This explains $n$ in the 200--600 range for near distances.
The peak at $E=10$ is due to $I$ being a single large interval
containing two saddles, one of which has a small width which demands
a small $h$.

We implemented the code in C with OpenMP and a MEX interface
(constructed via {\tt Mwrap}) to MATLAB
(version 2012b), and tested its speed on a desktop
workstation with two quad-core Intel Xeon E5-2643 CPUs at 3.3 GHz.%
\footnote{We also tested our codes on a laptop with
a quad-core Intel i7-3720QM at 2.6 GHz
and found speeds 70\%-100\% of those reported.}
\fref{f:time}(b) shows that at most $E$ values we achieve a mean rate
exceeding $10^5$ evaluations per second
(where we count $\Phi$ and its two derivatives as a single evaluation).
For near distances this drops to around 60\% of that.
Dips at various $E$ ranges are explained by the increased $n$.
The CPU time is believed to be dominated by calls to the complex exponential,
and arctangent, functions; memory usage is very small.

\bfi
\qquad
\ig{width=3in}{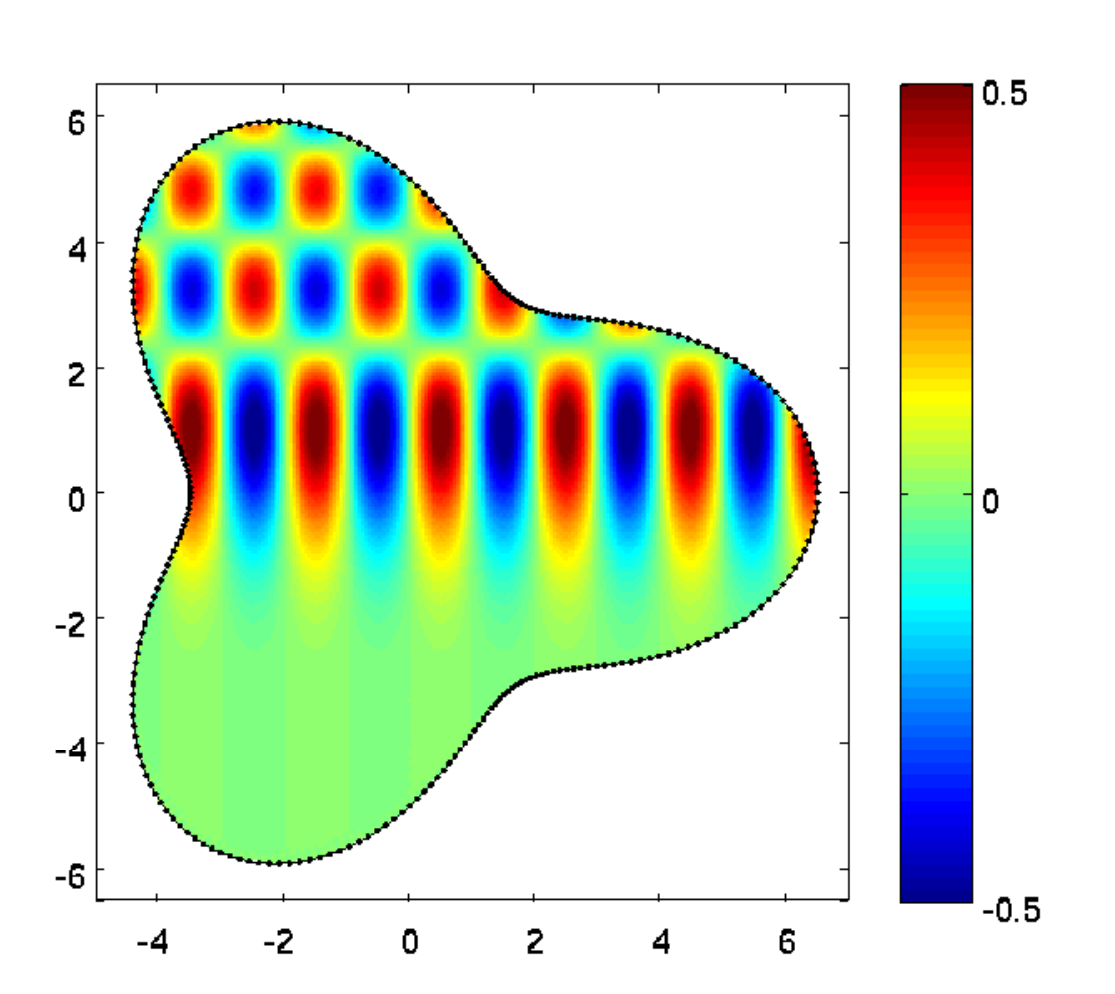}
\ig{width=3in}{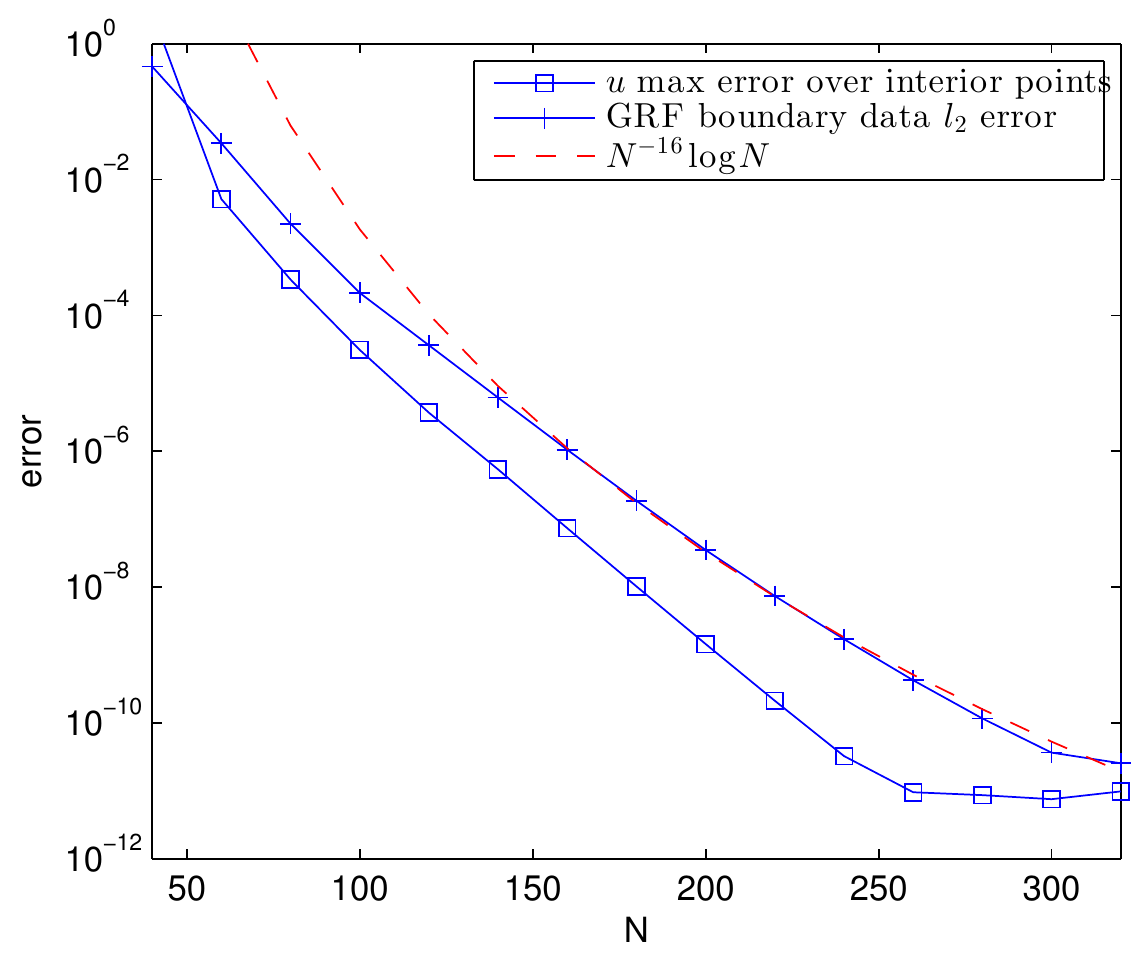}
\caption{
(a) Plot of interior Dirichlet solution $u$ as
evaluated by \eqref{ueval} given the density from solving the BIE,
with $N=260$ (boundary nodes $s_i$ shown as dots).
(b) Convergence of interior BVP solution:
maximum absolute error
($\square$ signs)
over 100 interior points chosen randomly to lie inside
a copy of $\pO$ scaled by 0.8, so points are not too close to $\pO$;
boundary error $\|S (\partial u/\partial n)^- - (D + \half I) u^-\|_{l_2}$
($+$ signs) for the Green's representation formula with the analytically known
data $\partial u/\partial n$ and $u$ on $\pO$. See \sref{s:intnum}.
\label{f:Airyconv}
}
\efi

% RRRRRRRRRRRRRRRRRRRRRRRRRRRRRRRRRRRRRRRRRRRRRRRRRRRRRRRRRRRRRRRRRRRRRRRRRRRRR
\section{Performance of the boundary value solver}
\label{s:num}

\subsection{Convergence for interior Dirichlet BVP}
\label{s:intnum}

To solve the interior BVP corresponding to \eqref{lhelm}--\eqref{ourk},
firstly
the parametrization of the curve $\pO$, and a number $N$ of boundary nodes,
is chosen.
Then the data vector $g_i = -2f(s_i)$, $i=1,\ldots,N$ is filled,
and the Nystr\"om matrix $A$ is filled using
\eqref{Asmooth} for entries away from the diagonal and
the Alpert correction of \sref{s:nyst} close to the diagonal,
with kernel
$K(t,s) = -2 \left(\partial \Phi(\zz(t),\zz(s))/\partial \nn_{\zz(s)}\right)
|\zz'(s)|$,
appropriate for the BIE \eqref{bieint}.
The dense linear system \eqref{linsys} is solved by direct Gaussian elimination
to get the density $\{\tau_j\}_{j=1}^N$, and the solution
evaluated by direct summation \eqref{ueval}.

We test convergence using Dirichlet data $f=u|_\pO$ coming from the
analytic separation of variables solution
\be
u(\xx) = \cos(\sqrt{E} x_1) \mbox{Ai}(-x_2)~,
\label{uex}
\ee
where Ai is the Airy function of the first kind, for $E=10$,
with $\pO$ a smooth ``trefoil'' domain
given by the polar function $r(\theta) = 5 + 1.5 \cos(3\theta)$,
about 5 wavelengths across.
\fref{f:Airyconv}(a) shows
the domain, boundary nodes, and resulting BIE solution constructed via
\eqref{ueval}.
In \fref{f:Airyconv}(b) we observe exponential convergence of the
absolute solution error
at interior points; we believe this rate is limited by the
distance of the nearest points to $\pO$ rather than the convergence of
the density. At $N=260$ we reach 11-digit accuracy (the solution $u$
has maximum size around $0.5$).
Filling $A$ took 5 seconds, and
the evaluation of $u$ at 32841 interior points used to plot \fref{f:Airyconv}(a)
took 70 seconds.

As an independent check of the discretization of the operators $S$ and $D$
on the boundary, by Green's representation formula
\cite[(2.5)]{coltonkress} \cite[Thm.~6.10]{mclean},
$$
u = {\cal S}  (\partial u/\partial n)^- - {\cal D}u^-
~,\qquad \mbox{ in }
\Omega~,
$$
and thus taking the evaluation point to $\pO$ from inside and applying
\eqref{jr1}--\eqref{jr2},
the boundary function $S (\partial u/\partial n)^- - (D + \half I) u^-$
should vanish. We show convergence of its norm in 
\fref{f:Airyconv}(b); it is consistent with the high order of the Alpert
scheme, and reaches 11-digit accuracy (each term, e.g.\
$(D + \half I) u^-$, has norm 1.5).

\begin{figure}[ht!] % ffffffffffffffffffffffffffffffffffffffffffffffff
% made via bie/fig_scattconv.m, 2nd option
\centering
\includegraphics[width=3.5in]{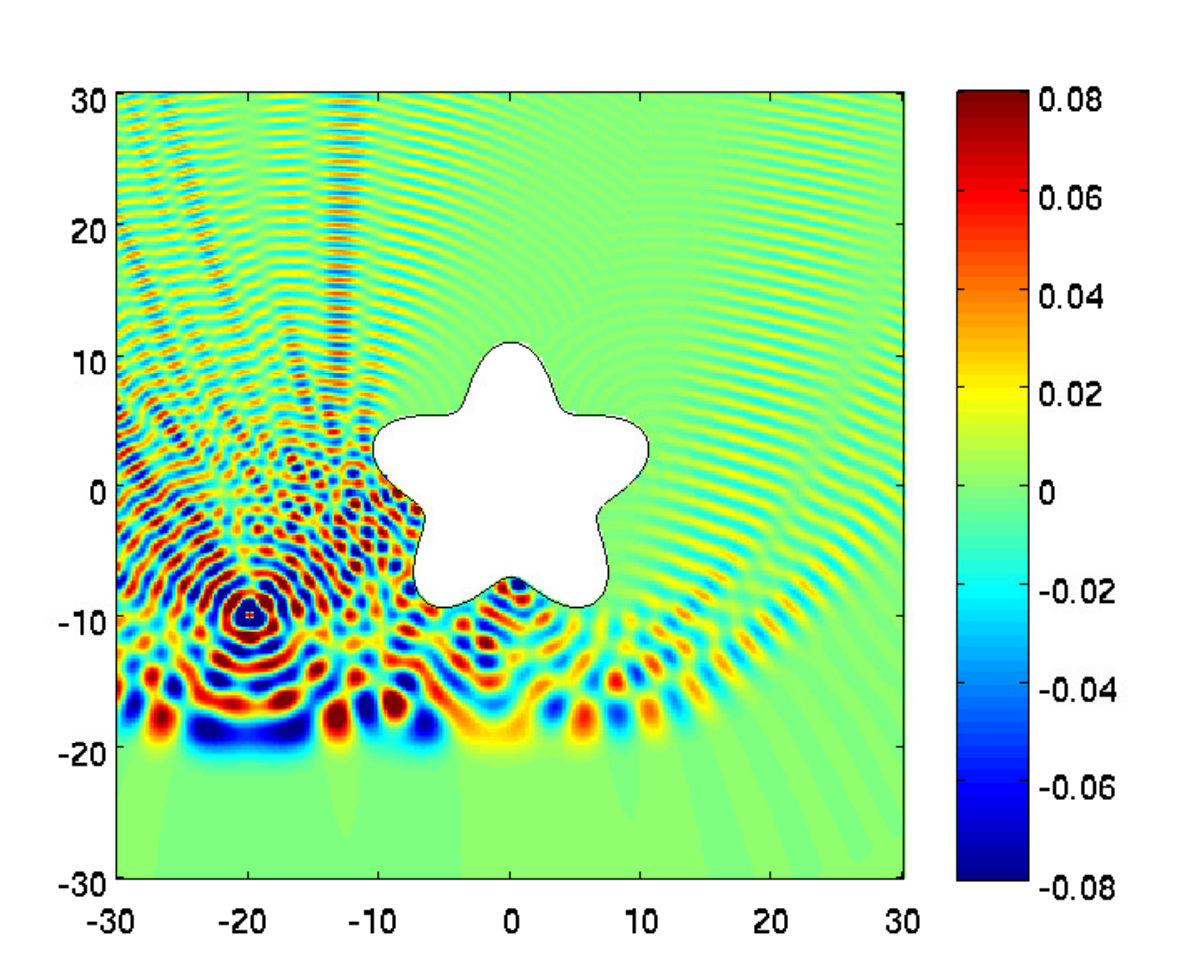}
\caption{Real part of total wave $u+\ui$ for a Dirichlet scattering problem at $E=20$,
with $\ui(\xx) = \Phi(\xx,\xx_s)$ with $\xx_s = (-20,-10)$.
Around 11 digit accuracy relative to the typical solution size
is achieved at $N=500$; see \sref{s:extnum}.
\label{f:scatt}
}
\end{figure}

\begin{table}[ht!]  % tttttttttttttttttttttttttttttttttttttttttttttttttttttttt
\hspace{.8in}
\begin{tabular}{ c | c | c | c | c}
$N$ & $A$ fill time	(s) & dense solve time (s)	&
evaluation time per target (s) & error
\\
\hline
% made via bie/fig_scattconv.m, first option
% molly data: (fricka is faster for fill, slower for eval, weirdly...
         200    &    3.8  &   0.004  &  0.0012  & 4.1e-05\\
          300   &    5.3 &    0.007  &  0.0018 &  6.3e-08\\
          400   &    7.9  &   0.019  &  0.0024 &  2.9e-10\\
          500   &    10.0  &   0.023  &  0.0030 &  2.6e-12\\
          600   &    12.6  &   0.028  &  0.0036  &         ---\\
\hline
\end{tabular}
% fricka:
%          200  &     3.0   &  0.0045 &   0.0016  & 4.1e-05\\
%          300   &    3.7  &   0.0099  &   0.0026  & 6.3e-08\\
%          400   &    6.4  &   0.015  &  0.0034  & 2.9e-10\\
%          500   &    8.1  &   0.022  &  0.0042  & 2.6e-12\\
%          600   &     10.2   &  0.036  &  0.0051  &   ---\\
\caption{
Convergence and timing for the small scattering problem shown in
\fref{f:scatt} and described in \sref{s:extnum}.
Evaluation time is for the solution $u$
via \eqref{ueval}, and is the mean value over a coarse grid covering the
region shown. % 3459 grid pts.
Error is the maximum absolute
error over 100 points lying uniformly on a circle of radius 12
(i.e.\ a closest distance of 1 from $\pO$),
estimated by comparing to the converged values for $N=600$.
\label{t:scatt}}
\end{table}

\begin{figure} % fffffffffffffffffffffffffffffffffffffffffffffffffff
% made via bie/fig_scattconv.m, 2nd option
\centering
\includegraphics[width=6in]{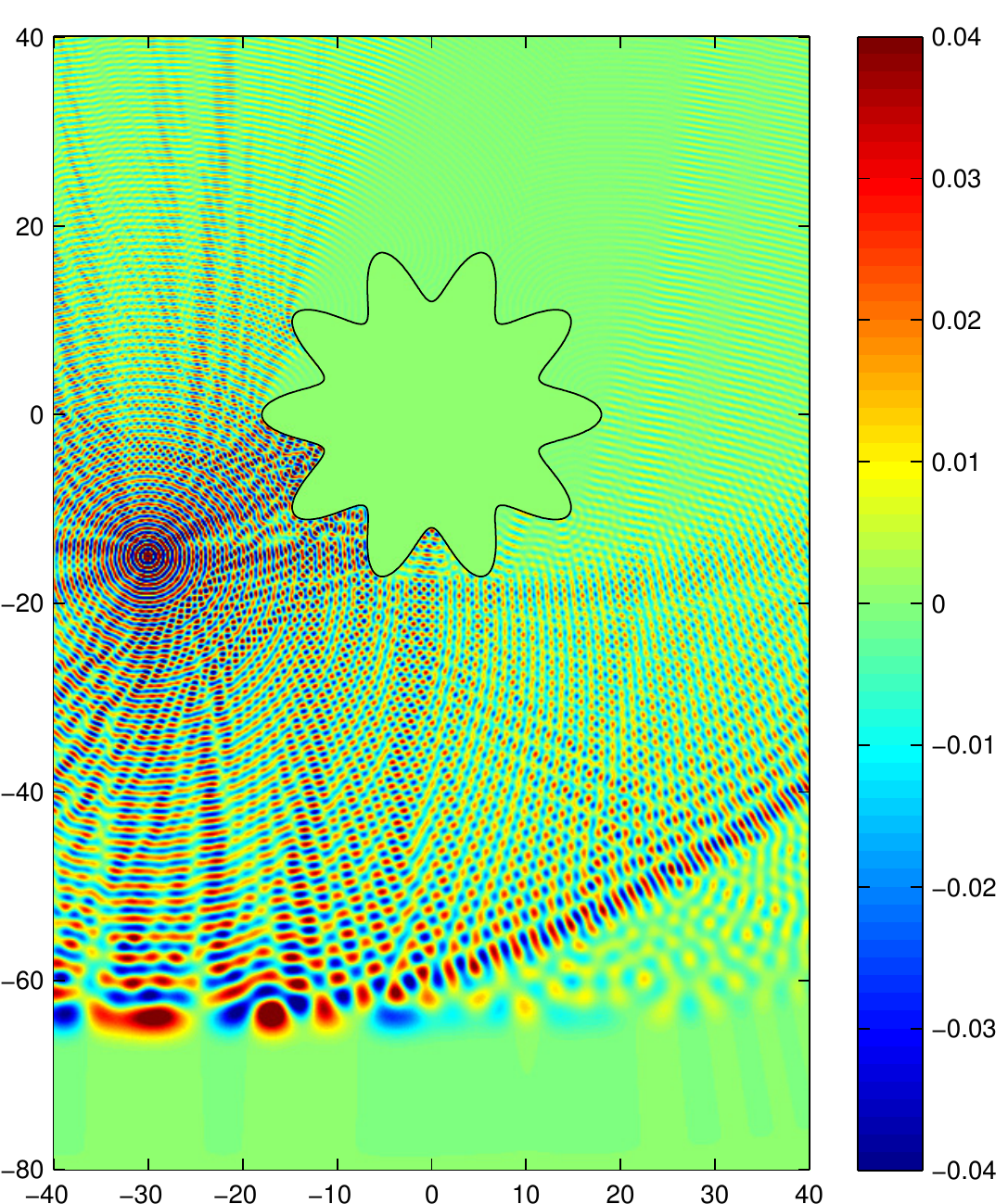}
\caption{Real part of total wave $u+\ui$ for a Dirichlet scattering problem at $E=65$,
with $\ui(\xx) = \Phi(\xx,\xx_s)$ with $\xx_s = (-30,-15)$.
Around 11 digit accuracy relative to the typical solution size
is achieved at $N=2000$; see \sref{s:extnum}.
}
\label{f:scattbig}
\end{figure}

\begin{table}  % tttttttttttttttttttttttttttttttttttttttttttttttttttttttt
\hspace{.8in}
\begin{tabular}{ c | c | c | c | c}
$N$ & $A$ fill time	(s) & dense solve time (s)	&
evaluation time per target (s) & error
\\
\hline
% made via bie/fig_scattconv.m, 2nd option
         1200    &   26   &  0.09  &  0.008 &  5.7e-08\\
         1600    &   37   &   0.20  &   0.011 &  2.1e-10\\
         2000    &   46   &   0.31  &   0.013 &  2.9e-12\\
         2400    &   58   &   0.42  &   0.016 &          ---\\
\hline
% brad's data from data_gen.m: - timings are same:
%ms: 1200  1600  2000  2400
%fill_times: 25.6697      36.9231      48.4374      60.9857
%solve_times: 0.10376     0.14524     0.25914     0.39403
%rec_times: 0.038863    0.041121     0.04277    0.042931
%errors: 3.1981e-08  1.5254e-10   2.285e-12           0
\end{tabular}
\caption{
Convergence and timing for the large scattering problem shown in
\fref{f:scattbig} and described in \sref{s:extnum}.
% eval time over 5856 coarse gridpts in -40:1:40 on both axes
% I believe Brad did -40:0.2:40 in x and -80:0.2:40 in y,
% est, 226400 pts.
Error is the maximum absolute
error over 100 points lying uniformly on a circle of radius 19
(i.e.\ a closest distance of 1 from $\pO$),
estimated by comparing to the converged values for $N=2400$.
\label{t:scattbig}}
\end{table}

% ssssssssssssssssssssssssssssssssssssssssssssssssssssssssssssssssss
\subsection{Convergence and timing for scattering problems}
\label{s:extnum}

For a scattering problem with given incident wave $\ui$,
as explained in the introduction,
the exterior BVP \eqref{lhelm}--\eqref{bc} is solved with $f=-\ui$.
We solve the combined-field
BIE \eqref{bie} similarly to the interior case summarized in
\sref{s:intnum},
except with data $g_i = 2f(s_i)$, $i=1,\ldots,N$
and kernel
$K(t,s) = 2 \left[ \partial \Phi(\zz(t),\zz(s))/\partial \nn_{\zz(s)} 
-i\eta\Phi(\zz(t),\zz(s)) \right]|\zz'(s)|$.
We test with two smooth scatterers which are chosen to be large
enough (diameter of order $E$)
that the wavelength has sizeable vertical variation across the object.

We first test a small example, at $E=20$, with shape
given by the polar function $r(\theta) = 9+2\sin(5\theta)$,
which is about 15 wavelengths across at the typical wavenumber $\sqrt{E}$.
The incident wave is due to a single nearby source at $\xx_s$.
The convergence in \tref{t:scatt} is consistent with exponential.
The solution time is entirely dominated by evaluations of $\Phi$,
and is consistent with %the around
$10^5$ evaluations per second.
%already observed.
The fill time has not yet reached its asymptotic $O(N^2)$, since
the $O(30 N)$ Alpert correction entries are expensive due to their small
source-target distances.
The dense linear system solve is $O(N^3)$, but insignificant in comparison.
A strict $O(N^2)$ overall scaling is recovered via using an
iterative solver; we applied GMRES \cite{gmres}
and found that 43 iterations were
required for  a residual of $10^{-12}$.
The total wave solution,
shown in \fref{f:scatt}, took 4 minutes to evaluate at
84089 grid points, i.e.\ around 350 target points per second.
Notice that the waves bend, and do not propagate below $x_2=-E=-20$.

Finally, we test a similar but more challenging case, at $E=65$,
with shape $r(\theta) = 15+3\cos(10\,\theta)$,
about 50 wavelengths across.
The convergence and timing
is in \tref{t:scattbig} and the total wave solution
is shown in \fref{f:scattbig}.%
\footnote{Curiously, fill times on the laptop were slightly faster than
for the desktop, but evaluation times were only 70\% as fast.}
Again, 11 digits of accuracy is achieved
at $N=2000$
(relative to the typical size of $u$, which is of order 0.1).
For GMRES, 59 iterations were needed to reach a residual of $10^{-12}$,
showing scarcely any growth from the lower-frequency example.
The plot in \fref{f:scattbig} took around 50 minutes
for 226000 target points, i.e.\ about 80 target points per second.
%One may see that
%waves do not propagate below $x_2=-65$, and 
The parabolic turning point
for the source is clearly visible, as well as waves of lower amplitude
that have been scattered and hence escape this parabola.

% cccccccccccccccccccccccccccccccccccccccccccccccccccccccccccccccccccccccccc
\section{Conclusion and discussion}
\label{s:conc}

We have presented an efficient scheme for high-frequency
scattering from
smooth objects embedded in a stratified medium in which the
inverse square of wave speed varies linearly in the vertical coordinate
(the ``gravity Helmholtz equation'').
Our high efficiency and accuracy comes from combining
numerical saddle point integration for an integral
representation of the fundamental solution $\Phi$,
with a boundary integral formulation and
high-order quadrature rules for the singular kernels,
allowing a problem 50 wavelengths in diameter to be solved to 11 digit
accuracy in less than a minute on a desktop or laptop.
Our detailed study of the saddle points (and their
connection to classical ray dynamics) allows
around $10^5$ evaluations of $\Phi$ per second, {\em independent}
of the wavenumber.
Solution cost is dominated by evaluations of $\Phi$,
which is trivially parallelizable,
and, once the matrix is filled, multiple incident waves at the same $E$
can be solved with negligible extra cost.
The scheme is strictly $O(N^2)$ when an iterative solver (such as GMRES) is
used; here convergence is rapid due to the second-kind formulation.
%although for any problem that fits on a workstation using an $O(N^3)$
%direct solver adds negligible cost.
%solve of $O(N^3)$ would not start to dominate until $N>10^5$.

In addition we placed the boundary value problem in the unbounded
stratified medium on a more rigorous footing by deriving radiation
conditions (\dref{d:rbcs}) such that the solution is unique.
%An %non-trivial
%analysis task 
It remains to prove %the empirical observation
the conjecture
that these are indeed satisfied by our causal $\Phi$;
this would give an existence proof for the BVP (\rref{r:exist}).

%Applications include modeling sound propagation in a temperature gradient,
%underwater acoustics, and graded-index optics.
In terms of future research,
the sound-hard and transmission problems \cite{coltonkress} are straightforward
variants, as %on what we have presented.
is the restriction to a half-space (reflected rays would need to be considered).
The BIE operators we have constructed
are also ideal for applying our medium's radiation boundary conditions to
finite-element solvers.
When the obstacle is no more than around 100 wavelengths across,
much acceleration is possible:
a kernel-independent fast multipole method (FMM)
\cite{kifmm} could be used to apply $A$ in each GMRES iteration,
or a fast direct solver \cite{m2011_1D_survey};
both would evaluate only $O(N)$ as opposed to $O(N^2)$ matrix elements.
The former would also be much faster than direct summation
for evaluation of $u$.
We hope that our numerical saddle point integration techniques
might prove useful for other (special) functions.
The generalization to 3D will be easy, since $\Phi$ may be expressed
directly using Airy functions \cite{bracher}.
A generalization to {\em quadratic} variation of the inverse
square wave speed is also possible
since the time-dependent Schr\"odinger Green's function is still
known analytically \cite{hellerhouches}; this could be used for
modeling guiding channels in underwater acoustics.

Documented C/OpenMP and MATLAB/MEX codes, with which all tests were performed,
are freely available at
{\tt http://math.dartmouth.edu/$\sim$ahb/software/lhelmfs.tgz}

\section*{Acknowledgements}
We have benefited from helpful discussion with
Simon Chandler-Wilde, Erik van Erp, and Nick Trefethen.
AHB is grateful for support from NSF grant DMS-1216656.
BJN is grateful for support from the
Paul K. Richter and Evalyn E. Cook Richter Memorial Fund.
The work of JMM and BJN was performed while at the Department of Mathematics
at Dartmouth College.

\appendix

%%%%%%%%%%%%%%%%%%%%%%%%%%%%%%%%%%%%%%%%%%%%%%%%%%%%%%%%%%%%%%%%%%%%%%%%%%%
%Radiation conditions for the gravity Helmholtz boundary value problem
\section{Proof of Theorem~\ref{t:rc}: uniqueness of radiative solutions}
\label{a:rc}

We adapt the radial methods of proof of \cite[Thm~3.7]{coltonkress}
to handle the very different asymptotic behaviors in horizontal and vertical
directions.
First we need the following Cartesian version of Rellich's far field
decay condition \cite[Lemma~2.11]{coltonkress}.
\begin{lem}[Cartesian Rellich]
Let $u$ satisfy \eqref{lhelm}, with medium \eqref{ourk}, in the complement
of a bounded domain $\Omega$,
and
\be
\lim_{x_2\to +\infty} k(x_2) \infint |u(x_1,x_2)|^2 dx_1 \; = \; 0
~.
\label{rell}
\ee
Then $u=0$ in $\eO$.
\label{l:rell}
\end{lem}
\begin{proof}
For sufficiently large $x_2$,
using the horizontal Fourier transform $\hat{u}(\xi,x_2) = \frac{1}{2\pi}
\infint u(x_1,x_2) e^{i\xi x_1} dx_1$,
the PDE becomes, for each $\xi\in\R$, an ODE in $x_2$,
$$
\partial_{x_2}^2 \hat{u}(\xi,x_2) + (x_2+E - \xi^2)\hat{u}(\xi,x_2) = 0
~.
$$
This is a shifted Airy's equation, thus, in terms of Airy functions
Ai and Bi,
$$
\hat{u}(\xi,x_2) = \alpha(\xi) \Ai (-x_2-E+\xi^2) + 
\beta(\xi) \Bi (-x_2-E+\xi^2)
$$
By unitarity of the Fourier transform, \eqref{rell} implies
\be
\lim_{x_2\to +\infty} k(x_2) \infint |\hat{u}(\xi,x_2)|^2 d\xi \; = \; 0
\label{rellft}
\ee
By the asymptotics $\Ai(-z) \sim \pi^{-1/2}z^{-1/4} \cos(2z^{3/2}/3 - \pi/4)$
and $\Bi(-z) \sim -\pi^{-1/2}z^{-1/4} \sin(2z^{3/2}/3 - \pi/4)$,
as $z\to+\infty$ \cite[(9.7.9), (9.7.11)]{dlmf},
and $k(x_2) \sim \sqrt{x_2}$,
it follows that if $\alpha(\xi)$ or $\beta(\xi)$ were nonzero on
any open subset of $\R$, then the limit \eqref{rellft} would be positive.
Thus $\alpha$ and $\beta$ are zero except possibly at a set of measure zero.
Taking the inverse Fourier transform,
$u(x_1,x_2) = 0$ for all $x_2$ sufficiently large.
Since \eqref{lhelm} has analytic coefficients, its solutions are analytic
in both variables.
By unique continuation, $u=0$ in all of $\eO$.
\end{proof}

Next we need flux conservation, which states that, for any bounded
region $D\subset\R^2$ with boundary $\partial D$ in which $u$ satisfies
\eqref{lhelm} with $k(x_2)^2$ real,
\be
- \im \int_{\partial D} u \overline{u_n}\, ds \; = \; 0 ~,
\label{flux}
\ee
where $u_n = \partial u/\partial n$ is the outward-pointing normal derivative.
The left-hand side may be interpreted as the wave energy flux
exiting the domain $D$.
This follows simply from taking the imaginary part of Green's first identity
$$
\int_{\partial D} u \overline{u_n} \, ds \; = \;
\int_D u \Delta \overline{u}  + |\nabla u|^2 \, d\xx
$$
after inserting $\Delta u = -k(x_2)^2 u$ from the PDE.

We now prove a result analogous to \cite[Thm.~2.12]{coltonkress}.
\begin{thm}[Non-negative incoming flux]
Let $\Omega\subset\R^2$ be a bounded domain. Let
$u$ solve \eqref{lhelm} with medium \eqref{ourk} in $\eO$,
be radiative according to Definition~\ref{d:rbcs}, and
have non-negative incoming flux, i.e.,
$$\im \int_\pO u \overline{u_n}\, ds \; \ge \; 0 ~.$$
Then $u=0$ in $\eO$.
\label{nonnegflux}
\end{thm}
\begin{proof}
%Let $R = [-M,M]\times[-L,L]$ be a rectangular region, with boundary $\pR$,
%containing $\Omega$.
Expanding the square in \eqref{UPRC} gives
$$
\lim_{x_2\to+\infty} \lim_{M\to\infty} \int_{-M}^M \frac{1}{k(x_2)}
\left|\frac{\partial u}{\partial x_2}\right|^2 + k(x_2)|u|^2 \,dx_1
+
2\im \int_{-M}^M u\frac{\partial \overline{u}}{\partial x_2} \, dx_1 
\; = \; 0~.
$$
Applying \eqref{flux} to the punctured rectangle
$(-M,M)\times(-x_2,x_2) \backslash \overline{\Omega}$,
by the decay conditions \eqref{DRC}--\eqref{LRRC} and Cauchy-Schwarz the
flux contributions from the bottom, left, and right sides vanish,
giving
$$
\lim_{x_2\to+\infty} \lim_{M\to\infty} \int_{-M}^M \frac{1}{k(x_2)}
\left|\frac{\partial u}{\partial x_2}\right|^2 + k(x_2)|u|^2 \,dx_1
 = - 2 \im \int_{\partial D} u \overline{u_n}\, ds
$$
analogous to \cite[(2.10)]{coltonkress}.
By the assumption of the theorem, the right-hand side is non-positive,
so \eqref{rell} holds, and Lemma~\ref{l:rell} completes the proof.
\end{proof}

Finally, to prove the uniqueness of the radiative solution to the
Dirichlet BVP \eqref{lhelm}--\eqref{bc},
we need only that if $u=0$ on $\pO$, and $u$ is a radiative solution,
then $u=0$ in $\eO$. Given the remark in the proof \cite[Thm~3.7]{coltonkress}
about the convergence of the normal derivative,
the incoming flux is zero and the result follows from Theorem~\ref{nonnegflux}.
We suspect that the above generalizes easily to more general
profiles $k(x_2)$.

% bbbbbbbbbbbbbbbbbbbbbbbbbbbbbbbbbbbbbbbbbbbbb
\bibliographystyle{elsarticle-num} 
\bibliography{alex}
\end{document}